\documentclass
[11pt]
{article}
\usepackage{amsmath,amssymb,amsthm
	,makeidx,fancyhdr
}
\usepackage{color}

\theoremstyle{definition}
\newtheorem{definition}{Definition}[section]
\newtheorem{example}[definition]{Example}

\theoremstyle{plain}
\newtheorem{theorem}[definition]{Theorem}
\newtheorem{conjecture}[definition]{Conjecture}
\newtheorem{proposition}[definition]{Proposition}
\newtheorem{lemma}[definition]{Lemma}
\newtheorem{corollary}[definition]{Corollary}
\newtheorem{question}[definition]{Question}

\newtheorem{remark}[definition]{Remark}
\newtheorem{fact}[definition]{Fact}

\newcommand{\game}{\mathcal{G}}
\newcommand{\qo}{\mathbb{Q}}
\newcommand{\ro}{\mathbb{R}}
\newcommand{\RR}{\mathbb{R}}

\newcommand{\Ult}{\operatorname{Ult}}

\newcommand{\fr}{{}^\frown}

\newcommand{\name}{\dot}
\newcommand{\can}{\check}
\newcommand{\force}{\Vdash}

\newcommand{\la}{\langle}
\newcommand{\ra}{\rangle}
\newcommand{\elem}{\prec}
\newcommand{\uhr}{\restriction}

\newcommand{\power}{\mathcal{P}}

\newcommand{\po}{\mathbb{P}}

\newcommand{\A}{\mathfrak{A}}

\newcommand{\U}{\vec{U}}
\newcommand{\W}{\vec{W}}
\newcommand{\x}{\vec{x}}

\newcommand{\MS}{\mathcal{MS}}

\newcommand{\CC}{\mathbb{C}}

\renewcommand{\NG}[1]{\Gamma_{#1}}
\newcommand{\G}[1]{G_{#1}}

\DeclareMathOperator{\cf}{cf}
\DeclareMathOperator{\cp}{cp}
\DeclareMathOperator{\supp}{supp}

\DeclareMathOperator{\Col}{Col}

\DeclareMathOperator{\SK}{SK}

\DeclareMathOperator{\suc}{succ}
\DeclareMathOperator{\ZFC}{ZFC}
\DeclareMathOperator{\NS}{NS}

\DeclareMathOperator{\cof}{Cof}
\DeclareMathOperator{\dom}{dom}
\DeclareMathOperator{\GCH}{GCH}

\DeclareMathOperator{\otp}{otp}

\DeclareMathOperator{\col}{Coll}

\newcounter{saveenumi}
\newcommand{\seti}{\setcounter{saveenumi}{\value{enumi}}}
\newcommand{\conti}{\setcounter{enumi}{\value{saveenumi}}}

\title{On Singular Stationarity I \\(mutual stationarity and ideal-based methods)}
\author{Omer Ben-Neria}
\date{}

\begin{document}
	\maketitle
  \begin{abstract}
	We study several ideal-based constructions in the context of singular stationarity.
	By combining methods of strong ideals, supercompact embeddings, and Prikry-type posets, we obtain three consistency results concerning mutually stationary sets, and answer a question of Foreman and Magidor (\cite{ForMag-MS})  concerning stationary sequences on the first uncountable cardinals, $\aleph_n$, $1 \leq n < \omega$.
\end{abstract}

\section{Introduction}\label{section-0}
This paper is the first of a two-part contribution to the project of generalized stationarity, and particularly to singular stationarity. 
In their seminal work on the non-stationary ideal on $\power_\kappa(\chi)$ (\cite{ForMag-MS}), Foreman and Magidor introduced the notion of mutual stationarity. 
\begin{definition}
	Let $R$ be a set of uncountable regular cardinals and $\vec{S} = \la S_\kappa \mid \kappa \in R\ra$ be a sequence of stationary sets such that $S_\kappa \subset \kappa$.
	The sequence $\vec{S}$ is mutually stationary if and only if for every algebra $\A$ on $\sup(R)$ there exists $M \elem \A$ such that $\sup(M \cap \kappa) \in S_\kappa$ for every $\kappa \in R \cap M$. 
\end{definition}
The first case which presents new challenges and applications to the study of stationary sets is when $R$ is countable and the sequence $\vec{S} = \la S_n \mid n < \omega\ra$ consists of stationary subsets $S_n \subset \kappa_n$ for an increasing sequence of cardinals $\vec{\kappa} = \la \kappa_n \mid n < \omega\ra$. 
The general question of which sequences $\vec{S}$ can be mutually stationary is connected with the long-standing problem whether arbitrary singular cardinals can be J\'{o}nsson.
In \cite{ForMag-MS}, the authors have identified the restricted question-of which sequences $\vec{S}$ of common and fixed cofinality can be mutually stationary-to be significantly important. 
They showed (in ZFC) that every sequence $\vec{S}$ with $S_\kappa \subset \kappa \cap \cof(\omega)$ is mutually stationary and used the result to prove that for every singular cardinal $\chi$, the generalized
nonstationary ideal on $\power_{\omega_1}(\chi)$  is not $\chi^+$-saturated.
They also proved that their mutual stationarity result for sets of countable cofinality ordinals cannot be extended to uncountable cofinalities by showing that in $L$, there is a sequence $\vec{S} = \la S_n \mid 2 \leq n < \omega\ra$ of stationary sets $S_n \subset \omega_n \cap \cof(\omega_1)$, which is not mutually stationary. 
This has prompted the following consistency question.
\begin{question}
	Is it consistent that every sequence $\la S_n \mid 2 \leq n < \omega\ra$ of stationary sets $S_n \subset \omega_n \cap \cof(\omega_1)$ is mutually stationary?
\end{question}

The main result of this paper provides a positive answer to this question.
\begin{theorem}\label{thm 1.1}
	It is consistent relative to the existence of infinitely many supercompact cardinals that every sequence of stationary subsets $S_n \subset \omega_n$ of some fixed cofinality is mutually stationary.
\end{theorem}

Prior to this, many consistency results have established various connections between mutually stationary sequences and large cardinals, via techniques of forcing and inner model theory. 
	Schindler (\cite{Sch-MS}) extended the mutual stationarity counter-example in $L$ to other core models $L[E]$ which can accommodate large cardinals at the level of overlapping extenders \footnote{i.e., the core model for almost linear iteration.}.
	In terms of consistency strength, the combined works of Cummings-Foreman-Magidor (\cite{CFM-CanStrII} for the upper bound) and Koepke-Welch (\cite{KoeWel-MS} for the lower bound) show that the existence of an $\omega$ sequence of cardinals $\vec{\kappa} = \la \kappa_n \mid n < \omega\ra$ such that every sequence $\vec{S}$ of sets $S_n \subset \kappa_n \cap \cof(\omega_1)$ is mutually stationary, is equiconsistent with the existence of a single measurable cardinal.
	Moreover, it is shown in \cite{CFM-CanStrII} that if $\vec{\kappa} = \la \kappa_n \mid n <\omega\ra$ is a Prikry generic sequence, then every sequence of stationary sets on it is mutually stationary. 
	Building on this model, Koepke (\cite{Koe-MS}) has devised an additional forcing extension in which the cardinals in $\vec{\kappa}$ become the cardinals $\aleph_{2n+1}$, $n < \omega$, and every stationary sequence $S_n \subset \aleph_{2n+1} \cap \cof(\omega_1)$ is mutually stationary. 
	Koepke's result is consistency-wise optimal in several different senses. First of all, the large cardinal assumption of a single measurable cardinal does not suffice to remove the gaps between the cardinals, as Koepke and Welch have shown (\cite{KoeWel-MS2}) that the mutual stationarity of every sequence $\vec{S}$ with $S_n \subset \omega_n \cap \cof(\omega_1)$ implies there exists an inner model with many measurable cardinals of high Mitchell order\footnote{i.e., there exists a sequence $\la \kappa_n \mid n < \omega\ra$ of regular cardinals, such that for almost every $m < \omega$, the set of cardinal $\alpha < \kappa_m$ whose Mitchell order is at least $\kappa_{m-2}$ is stationary in $\kappa_m$.}. Second, Koepke and Welch have further shown that a change in cofinalities of $S_n \subset \omega_{2n+1} \cap \cof(\omega_1)$, from $\omega_1$ to $\omega_2$, requires to increase the large cardinal assumption to at least inaccessibly many measurable cardinals. 
	Nevertheless, Koepke's result indicates that the consistency strength of mutual stationarity of stationary subsets of the odd $\aleph_n$'s should be weaker than the strength of the result of Theorem \ref{thm 1.1}. We prove that this is indeed the case: Let us say that a cardinal $\kappa$ carries a $(*,\lambda)$-sequence of measures if there is a Mitchell order increasing sequence of $\kappa^+$-supercompact measures on $\kappa$, of length $\lambda$.
	\begin{theorem}\label{thm1.3}
		The assertion that every sequence of stationary sets $S_n \subset \omega_{2n+1}$ of common fixed cofinality is consistent relative to existence of infinitely many cardinals $\la \kappa_n \mid n < \omega\ra$ such that each $\kappa_n$ carries a $(*,\kappa_{n-1}^+)$ sequence of measures\footnote{For $n = 0$, we replace $\kappa_{n-1}^+$ with $\aleph_0$}.
	\end{theorem} 
	Although the consistency assumption of Theorem \ref{thm1.3} is at the level of supercompact cardinals, we expect the arguments of the proof to lead to the following tighter consistency assumption, at the level of measurable cardinals.
	\begin{conjecture}\label{con1.35}
		The assertion that every sequence of stationary sets $S_n \subset \omega_{2n+1}$ of common fixed cofinality, is consistent relative
		a sequence of infinitely many  measurable cardinals $\la \kappa_n \mid n < \omega\ra$ such that each $\kappa_n$ carries an $(\omega,\kappa_{n-1}^+)$ repeat point measure.
	\end{conjecture}
	The notion of an $(\omega,\lambda)$ repeat point measure on a cardinal $\kappa$ has been introduced by Gitik in (\cite{Gitik-NSprec}), and is strictly 
	weaker than the assumption of $o(\kappa) = \kappa^{++}$.
	The assumptions of Theorem \ref{thm1.3} and the conjecture are related through the work of Gitik on precipitous nonstationary ideals. 
	The proof of Theorem \ref{thm1.3} establishes a new connection between Prikry-type forcing, strong ideals, and mutual stationarity, through the notion of Prikry-closed ideal, which we believe to be of independent interest.     
	It is well known that come consistency results in set theory rely on Prikry-type forcing techniques, and we expect that the ability to perform such forcing ``tasks``  while obtaining strong Prikry-types of ideals could have many other applications. 
	
	In addition to the results which focus on stationary sets of some common fixed cofinality, other consistency results establish the mutual stationarity of sequences which contain finitely many cofinalities. These are described in the works of Liu-Shelah (\cite{LiuShe-MS}) and  Adolf-Cox-Welch (\cite{AdoCoxWel-MS}). 
	The stationary sequences which are studied in the two papers consist of stationary sets of full cofinality. Namely, sequences $\vec{S} = \la S_n \mid n < \omega\ra$ for which each $S_n$ is of the form $S_n = \kappa_n \cap \cof(\mu_n)$ for some $\mu_n < \kappa_n$. 
	The paper \cite{LiuShe-MS} contains a theorem of Shelah which shows that from the assumption of infinitely many supercompact cardinals, it is consistent that for every $k < \omega$ and a function $f : \omega\setminus k \to \{\omega,\omega_k\}$, the sequence $\vec{S}$ defined by $S_n = \omega_n \cap \cof(f(n))$ is mutually stationary. 
	We improve the result and show it is possible to replace the full-cofinality sets with arbitrary stationary subsets.
	
	\begin{theorem}\label{thm1.2}
		It is consistent relative to the existence of infinitely many supercompact cardinals that for every $k < \omega$, for every sequence $\vec{S} = \la S_n \mid k<  n < \omega\ra$ of stationary sets $S_n \subset \omega_n$ which consist of ordinals of cofinalities $\omega$ or $\omega_k$, $\vec{S}$ is  mutually stationary.
	\end{theorem}

	\textbf{Organization of the paper - }  Following a review of some preliminaries below, the rest of the paper is organized as follows: In the first part of Section \ref{section-1}, we describe the basic method by which the stationary witnessing structures $M \elem \A$ are obtained. The method is based on the existence of certain nonstationary ideals. In the second part of that section, we prove Theorem \ref{thm 1.1} by constructing such ideals from elementary embeddings associated with supercompact cardinals. In Section \ref{section-2}, we build on the results of the previous section and combine them with arguments of Foreman and Magidor to prove Theorem \ref{thm1.2}. The first sections require only basic knowledge of forcing and large cardinal theory, which can be found in most introductory textbooks\footnote{e.g., see \cite{jech}.}.
	In Section \ref{section-3}, we prove Theorem \ref{thm1.3}. The proof requires basic familiarity with Prikry-type posets and their iteration.
	Although the arguments rely on methods of Gitik from \cite{Gitik-HB},\cite{Gitik-iteration}, and \cite{Gitik-precaleph2}, which will not be reconstructed in full,  we have tried to include many details describing Gitik's work and ideas, to allow a continuous line of arguments and thoughts.

	\subsection{Preliminaries}
	We review some relevant background material on algebras, partially order sets, and ideals. 
	Our notations are (hopefully) standard, and we list several notable ones. For a regular cardinal $\lambda$, we denote by $\cof(\lambda)$ the class of all ordinals  with cofinality $\lambda$, and similarly use notations such as $\cof(<\lambda)$ and $\cof(\geq \lambda)$ in the obvious way. 
	
	\textbf{Algebras and stationary sets - }
	An algebra on a set $H$ is a structure $\A = \la H,f_i\ra_{i<\omega}$ where each $f_i$ is a finitary function defined on $H$. 
	By a substructure $M$ of $\A$ we will always mean that $M \elem \A$ is an elementary substructure. For every subset $X$ of $\A$, we denote 
	the Skolem hull of $X$ in $\A$ by $\SK_{\A}(X) = \{ t^{\A}(p) \mid t^{\A} \text{ 
		is a } \A\text{-Skolem term, and } p \in  X\}$.
	We say that a substructure $M$ is $\mu$-closed if ${}^{<\mu}M \subset M$.
	A standard argument shows that a subset $S$ of a regular cardinal $\kappa$ is stationary if and only if for every set $H$ which contains $\kappa$ and algebra $\A$ on $H$, there exists a substructure $M \elem \A$ such that $\sup(M \cap \kappa) \in S$. Therefore if $\vec{S} = \la S_n \mid n < \omega\ra$ is mutually stationary then each $S_n$ is stataionary. 
	Since every algebra on a set $H$ can be easily extended to an algebra on a larger domain with the same effect on stationarity, we can replace the domain of the algebra in the definition of mutually stationary sets from $\sup(R)$ to $H_\theta$ for any $\theta > \sup(R)$. 
	Suppose that $M$ is a substructure of an algebra $\A$ of size $|M| = \mu$, and $\kappa_1 < \kappa_2 < \dots < \kappa_m$ are regular caridnals in $M$ above $\mu$. By a well-known Lemma \cite{Baum}, adding to $M$ ordinals below some $\kappa_i$ does not change its supremum below regular cardinals $\kappa > \kappa_i$. Therefore, it is not difficult to verify that
	for every finite sequence of stataionary sets $S_1,\dots,S_m$ with $S_i \subset \kappa_i \cap \cof(\leq \mu)$, there exists an elementary extension $M'$ of $M$ (i.e., $M \elem M' \elem \A$) so that $\sup(M' \cap \kappa_i) \in S_i$ for all $i = 1,\dots,m$, and $\sup(M' \cap \kappa) = \sup(M \cap \kappa)$ for all regular cardinals $\kappa > \kappa_m$ in $M$ (for more details, see \cite{ForMag-MS}). This argument allows us to ignore a finite initial segment of a stationary sequences $\la S_n \mid n < \omega\ra$ of bounded cofinalities and argue for the mutual stationarity of a tail. 
	
	\textbf{Partially ordered sets and projections - }
	A partially ordered set $\po$ (referred to as a poset, or a forcing notion) consists of an order relation $\leq$ on a domain, whose notation will usually be abused and denoted by $\po$ as well. We shall use the Jerusalem forcing convention by which for two conditions $p,q \in \po$,  $p$ is stronger (more informative) than $q$ is denoted by $p \geq q$. A generic filter for $\po$ will usually be denoted by $\G{\po}$. 
	We say that a forcing $\po$ is $\mu$-closed if every $\delta < \mu$ and an increasing sequence of conditions $\la p_i \mid i < \delta\ra$ of $\po$, there exists a condition $p \in \po$ which is an upper bound to the sequence. $\po$ is called $\mu$-distributive if it does not add new sequences of ordinals of length $\delta < \mu$. 
	Let $\po$ and $\qo$ be partially ordered sets. We say that $\po$ absorbs $\qo$, or that $\po$ projects to $\qo$ if the forcing with $\po$ introduces a generic filter of $\qo$.
	Given a $\po$-name $\NG{\qo}$ of a $\qo$-generic introduces a forcing projection $\pi_{\G{\qo}}$ from $\po$ to the boolean completion of $\qo$, defined by $\pi_{\NG{\qo}}(p) = \{ q\in \qo \mid p \not\force \can{q} \in \NG{\qo}\}$. 
	In a generic extension $V[\G{\qo}]$ by a generic filter $\G{\qo}$, we define the quotient of $\po$ with respect to $\NG{\qo}$ and $\G{\qo}$ to be the poset $\{ p \in \po \mid \pi_{\NG{\qo}}(p) \in \G{\qo} \}$.
	In general, this quotient is denoted by $\po/\la \NG{\qo},\G{\qo}\ra$, however in many standard cases, where the projection $\pi_{\NG{\qo}}$ is natural\footnote{e.g., if $\po = \la \po_{\alpha}, \po(\alpha) \mid \alpha < \kappa\ra$  is an iterated forcing and $\qo = \qo_\nu$ is an initial segment of this iteration} we omit $\NG{\qo}$ from the notation and write $\po/\G{\qo}$ for the quotient. 
	Further, when working in the ground model $V$, we shall denote the $\qo$-name for this quotient poset of $\po$ by $\po/\qo$, and frequently identify $\po$ with its isomorphic forcing notion $\qo * \po/\qo$.
	
	
	\textbf{Ideals -}
	A $\kappa$-complete ideal on $\kappa$ is a subset $I$ of the powerset of $\kappa$, $\power(\kappa)$, which is closed under subsets and unions of less than $\kappa$ of its sets. We shall always assume $I$ is nonprincipal and uniform, namely $I \neq \emptyset$, $\kappa \not\in I$, and $\alpha \in I$ for every $\alpha < \kappa$. 
	Given an ideal $I$ on $\kappa$, we denote its dual filter $\{\kappa \setminus Z \mid Z \in I\}$ by $\check{I}$, and the family of $I$-positive sets, $\power(\kappa) \setminus I$,  by $I^{+}$. 
	Given two $I$-positive sets $A,B \in I^+$, we write $A \leq_I B$ when $A\setminus B \in I$, and $A \equiv_I B$ when $A \leq_I B$ and $B \leq_I A$.

	\section{Mutual stationarity and closed nonstationary ideals}\label{section-1}
	
	The purpose of this section is to prove Theorem \ref{thm 1.1};
	we show that starting from a model with infinitely many supercompact cardinals, there exists a generic extension in which for every algebra $\A$ on $H_\theta$ for some regular $\theta > \aleph_\omega$ and a stationary sequence $\vec{S} = \la S_n \mid k < n < \omega\ra$ in the $\aleph_n$'s, of fixed uncountable cofinality $S_n \subset \omega_n \cap \cof(\omega_k)$, there exists a substructure $M \elem \A$ satisfying $\sup(M \cap \kappa_n) \in S_n$ for all $n > k+1$. 
	
	In order to establish the existence of a desirable substructure $M$ in the generic extension, we shall construct an elementary increasing sequence $\la M_n \mid k+1 < n < \omega\ra$ of substructure of $\A$, which are all $\omega_k$-closed, have size $\aleph_k$, and further satisfy the following two conditions:
	\begin{itemize}
		\item $\sup(M_n \cap \kappa_n) \in S_n$,
		\item $M_{n+1} \cap \kappa_n = M_n \cap \kappa_n$.
	\end{itemize}
	
	It clearly follows that $M = \cup_n M_n$ is a substructure of $\A$ with $\sup(M \cap \omega_n) \in S_n$ for all $n > k+1$. As mentioned in the preliminary section, we can further extend $M$ below $\kappa_{k+1}$ (without changing its supermum below larger cardinals) and obtain a substructure of $\A$ which further meets $S_{k+1}$.
	The structures $M_n$ are defined inductively starting from any $\omega_k$-closed structure $M_k \elem \A$ of size $|M_k| = \aleph_k$, which contains the sequence $\vec{S}$.
	The existence of a $\aleph_k$-closed substructures of an arbitrary algebra $\A$ is guaranteed in our model which satisfies the $\GCH$ holds below $\aleph_\omega$.
	We note that the construction framework described below, can be applied to internally approachable structures, whose existence does not require the cardinal arithmetic assumption of $\GCH$, and to stationary sets which belong to the approachability ideal (see \cite{CFM-CanStrII} for additional information).

	The inductive step (producing $M_{n+1}$ from a given $M_n$) relies on the ability to prove that the following holds in our model. \\
	\textbf{ The end-extension property -} Let $\mu \leq \lambda < \kappa$ be three regular cardinals in $\A \cap \aleph_{\omega}$. 
	Suppose that $M \elem \A$ is $\mu$-closed, has size $|M| = \mu$, and contains $\mu,\lambda,\kappa$, then for every stationary set $S \subset \kappa \cap \cof(\mu)$ in $M$, $M$ has an elementary extension $N$,  $M \elem N \elem \A$, such that $\sup(N \cap \kappa) \in S$ and $N \cap \lambda = M \cap \lambda$.\\
	
	The rest of this section is organized as follows. We commence by describing a framework for constructing end extensions. The main assumption allowing this construction is the existence of sufficiently many nonstationary and closed ideals on $\kappa$ so that every stationary subset $S$ of $\kappa$ is positive with respect to such an ideal.
	Following the description of nonstationary closed ideals and their connection to end-extensions of substructures, we show that the relevant ideal assumption holds in a generic extension model, in which supercompact cardinals are collapsed to be the $\aleph_n$'s. 
	
	\subsection{Supercompact cardinals and nonstationary closed ideals}
	Suppose that $\mu, \lambda < \kappa$ are all regular cardinals, and  $\A$ be  an algebra which extends $\la H_\theta,\in,<_\theta\ra$, where  $\theta > 2^\kappa$ is regular and $<_\theta$ is a well-order on $H_\theta$.
	Let $M \elem \A$ be a $\mu$-closed substructure of size $|M| = \mu$, which contains $\mu,\lambda$, and $\kappa$. 
	To find a suitable elementary end-extension $M$ above $\lambda$, we shall consider elementary extensions of the form $N = \SK_{\A}(M \cup X) = \{ t^{\A}(p) \mid t^{\A} \text{ is a } \A\text{-Skolem term, and } p \in M \cup X\}$.
	
	The following folklore result follows from the fact $H_\theta$ satisfies the axiom of replacement and separation, and therefore the restriction of an $\A$-Skolem term $t$ to $H_\kappa$ must be an element of $\A$. 
	
	\begin{fact}\label{Lem 2.1}
		If $X$ is a subset of $H_\kappa$ then
		\[\SK_{\A}(M \cup X) = \{ f(\vec{z}) \mid \vec{z} \in X^{<\omega}, f \in M, \dom(f) = [H_\kappa]^{|\vec{z}|} \}.\] 
	\end{fact}

	\begin{definition}\label{Def 2.1}
		Let $\kappa \in M$ be a regular cardinal and $A \subset$ be an unbounded set. 
		We say $A$ is $\lambda$-\textbf{homogeneous} for $M$ if $f\uhr A$ is constant for every function $f : \kappa \to \lambda$ in $M$. 
		We further say $A$ is \textbf{approximated} in $M$ if for every function $f : \kappa \to \lambda$ in $M$ there exists $A_f \in M$ such that $A \subset A_f \subset \kappa$ and $f\uhr A_f$ is constant. 
	\end{definition}
	
	The following is immediate from the preceding fact and definition.
	\begin{lemma}\label{Lem 2.15}
		If $A$ is a $\lambda$-homogeneous and approximated set in $M$ then for every $\alpha \in A$, $\SK_{\A}(M \cup\{\alpha\})$ does not add new ordinals below $\lambda$ to $M$, and is therefore an end-extension of $M$ above $\lambda$.
	\end{lemma}

	The following folklore example demonstrates how large cardinal assumptions can be used to construct homogeneous approximated sets. 
	\begin{example}
		Suppose that $\kappa$ is a measurable cardinal and $U \in M$ is a $\kappa$-complete ultrafilter on $\kappa$. Let $A = \cap( U \cap M)$. Then for every $\lambda \in M \cap \kappa$, $A$ is a $\lambda$-homogeneous approximated set for $M$. 
	\end{example}
	\begin{proof}
		First, note that $A \in U$, since $U$ is $\kappa$-complete and $|U \cap M| < \kappa$. Next, for every $\lambda < \kappa$ and $f : \kappa \to \lambda$, $\kappa$ decomposes into a disjoint union $\kappa = \biguplus_{\nu<\lambda}f^{-1}(\nu)$, and exactly one of the sets $f^{-1}(\nu^*) $ must be a member of $U$. Since $U,f \in M$ then $\nu^* \in M$, and so, the set $A_f =  f^{-1}(\nu^*) \in M \cap U$ is a suitable approximation of $A$ with respect to $f$.         
	\end{proof}

	When $\kappa$ is accessible (e.g., $\kappa = \aleph_n$ for some $n  <\omega$), one can replace the ultrafilter-based method in the example above with an ideal-based construction.
	Suppose that $I$ is an ideal on $\kappa$ and $f : \kappa \to \lambda$. If $I$ is not prime (i.e., if $\check{I}$ is not an ultrafilter) it is not guaranteed that one of the decomposition sets of $\kappa = \biguplus_{\nu < \lambda}f^{-1}(\nu)$ belongs to $\check{I}$. Nevertheless, since $I$ is $\kappa$-complete and $\lambda < \kappa$, at least one of the sets $f^{-1}(\nu)$, $\nu < \lambda$, is a member of $I^+$. Moreover, if $I,f \in M$ and $B \in I^+ \cap M$, then there exists $\nu^* < \lambda$ such that $f^{-1}(\nu^*) \cap B \in I^+$.
	
	\begin{definition}\label{Def 2.15}
		We say that a non-principal $\kappa$-complete ideal $I$ on $\kappa$ is $\mu$-closed if $I^+$ has a $\leq_I$ dense subset $D$ such that the restriction of $\leq_I$ to $D$ is $\mu$-closed.
	\end{definition}
	
	The following well-known lemma has been used to obtain various results from certain ideal-based assumptions. In the context of mutual stationarity, it has been used in \cite{CFM-CanStrII} to reprove a theorem of Shelah about mutual stationarity in two cofinalities. 
	
	\begin{lemma}\label{Lem 2.2}
		Let $M \elem \A$ be a $|M|$-closed algebra.  Suppose that $\kappa > |M|$ is a regular cardinal and  $I \in M$ is a $\kappa$-complete, $(|M|+1)$-closed ideal on $\kappa$. Then for every $A \in I^+ \cap M$ and $\lambda \in M \cap \kappa$ there exists a subset $B \subset A$ in $I^+$ which is $\lambda$-homogeneous and approximated in $M$. 
	\end{lemma}
	\begin{proof}
		Let $D \in M$  be a $\leq_I$-dense set of $I^+$ which contains lower bounds to all  $\leq_I$ decreasing sequences of length $\delta \leq |M|$ of its elements. 
		Fix an enumeration $\vec{f} = \la f_i \mid i < |M|\ra$ of all the functions $f: \kappa \to \lambda$ in $M$ and define a $\leq_I$ decreasing sequence of sets
		$\vec{A} = \la A_i \mid i < |M|\ra \subset D \cap M$. Each $A_i$ will be definable from $\vec{f}\uhr i, \vec{A}\uhr i$, $D$, and the definable well order $<_\theta$ of $\A$. The fact $M$ is $|M|$-closed implies that every initial segment of $\vec{f}$ belongs to $M$. This, in turn, implies $\vec{A}\uhr i, A_i \in M$ for all $i <|M|$.
		We first pick $A_0 \in D \cap M$ to be a subset of $A$. Suppose that $\vec{A}\uhr i$ has been constructed for some $i > 0$. We assume $\vec{A}\uhr i \subset M \cap D$ is $\leq_I$-decreasing. If $i$ is a limit ordinal we take $A_i \in D$ to be the first $\leq_I$-lower bound to $\vec{A}\uhr i$ according in the well order $<_\theta$. 
		Suppose $i = i' + 1$ is a successor ordinal.     
		By the remark preceding to Definition \ref{Def 2.15} above, there exists some $\nu < \lambda$ in $M$ such that $f_{i'}^{-1}(\nu)\cap A_{i'} \in I^+$.
		Let $\nu' < \lambda$ be the minimal value $\nu$ for which the last holds, and let $A_i$ be the $<_\theta$-first set in $D$ such that $A_i \leq_I f_{i'}^{-1}(\nu')\cap A_{i'}$.
		
		This concludes the construction of the sequence $\vec{A} \subset M \cap D$.  It is immediate from the successor steps of the construction that for every $i' < |M|$ there exists a set $Z_{i'} \in M \cap I$ such that $f_{i'}\uhr (A_{i'+1} \setminus Z_{i'})$ takes a constant value $\nu' \in M$.
		Let $B' \in D$ be a $\leq_{I}$-lower bound of $\vec{A}$,  $Z' = \bigcup_{i' < |M|}Z_{i'}$, and define  $B = (A \cap B') \setminus Z'$. 
		Since $Z' \in I$ and $B' \leq_I A$, we have that $B \subset A$ is $I$-positive, and it is clear from the construction that $B$ is a $\lambda$-homogeneous and approximated in $M$.
	\end{proof}
	
	By Lemma \ref{Lem 2.2} we see that if $S \in M \cap I^+$ for $M,I$ which satisfy the conditions of the Lemma, then there exists $A \subset S$ in $I^+$ which is  $\lambda$-homogeneous and approximated in $M$. Hence, for every $\alpha \in A \setminus \sup(M \cap \kappa)$, the structure $M^* = \SK_{\A}(M \cup \{\alpha\})$  is an elementary end-extension of $M$ above $\lambda$. Of course $\alpha \in M^*$ cannot be $\sup(M^* \cap \kappa)$ and there is no reason to believe $\sup(M^* \cap \kappa) \in S$.
	To obtain a desirable end-extension of $M$ which meets $S$, we  replace $M^*$ with a substructure $N \elem M^*$ satisfying $\sup(N \cap \kappa) =\alpha$. 
	We will now see that this is possible, subject to the additional assumption that $I$ is nonstationary. The following notions are needed to construct $N$. 
	
	\begin{definition}\label{def-2.25}
		${}$
		\begin{enumerate}
			\item A ladder system on a cardinal $\kappa$ is a function $\delta : \kappa \to [\kappa]^{<\kappa}$ which assigns to each limit ordinal $\alpha < \kappa$ an increasing cofinal set $\delta(\alpha) = \la \delta(\alpha)(i) \mid i < \cf(\alpha)\ra$  in $\alpha$.
			\item Suppose that $\delta$ is a ladder system on $\kappa$ and let  $s :\kappa \to H_\kappa$ be the $\delta$-derived initial segments function, defined by $s(\alpha) = \{ \delta(\alpha)\uhr j \mid j < \cf(\alpha)\}$ where $\delta(\alpha)\uhr j = \la \delta(\alpha)(i) \mid i < j\ra$ for each $j < \cf(\alpha)$.  
		\end{enumerate}
	\end{definition}

	Let $\delta$ be a ladder system on $\kappa$ in $M$. 
	It follows that both $\delta(\alpha)$ and $s(\alpha)$ belong to $M^* = \SK_{\A}(M \cup \{\alpha\})$. Consequently, if we define 
	$N = \SK_{\A}(M \cup s(\alpha))$ then $M \elem N \elem M^*$ and thus
	$N \cap \lambda = M \cap \lambda$. 
	
	
	\begin{lemma}[$\GCH$]\label{Lem 2.3}
		Let $\alpha \in \kappa\setminus \sup(M \cap \kappa)$ such that $\cf(\alpha) \subset M$ and define $N =  \SK_{\A}(M \cup s(\alpha))$. 
		\begin{enumerate}
			\item If $M$ is $\cf(\alpha)$-closed then so is $N$.
			\item If $\alpha \in \kappa\setminus \sup(M \cap \kappa)$ belongs to every closed unbounded subset of $\kappa$ in $M$ then
			$\alpha = \sup( N \cap \kappa)$. 
		\end{enumerate}
	\end{lemma}
	\begin{proof}${}$
		\begin{enumerate}
			\item This is a direct consequence of Fact \ref{Lem 2.1} and the closure properties of $M$ and $s(\alpha)$.
			Every element of $N = \SK_{\A}(M \cup s(\alpha))$ is of the form
			$f(\vec{z})$, where $f : H_\kappa \to \kappa$ belongs to $M$ and $\vec{z} \subset [s(\alpha)]^{<\omega}$. Moreover, as $\cf(\alpha) \subset M$ and the elements of $s(\alpha)$ form a $\subset-$chain of length $\cf(\alpha)$, we have that every sequence of less than $\cf(\alpha)$ elements in $s(\alpha)$ is definable in $M$ from an element $s(\alpha)$. It follows that every sequence of length $\delta < \cf(\alpha)$ of elements in $N$ is definable from a sequence of length $\delta$ of functions in $M$, and from a single element in $s(\alpha)$. The second is in $N$ by its definition, and the first is in $N$ because $M$ is $\cf(\alpha)$-closed. Hence $N$ is $\cf(\alpha)$-closed.
			
			\item 
			The fact $\cf(\alpha) \subset M$ guarantees $s(\alpha) \subset N$, which in turn, implies $\delta(\alpha) \subset N$, and thus, that $\alpha \geq \sup(N \cap \kappa)$. Next, suppose $\gamma \in N \cap \kappa$. Then $\gamma$ is of the form 
			$\gamma  = f(\delta(\alpha)\uhr j)$ for some $f : H_\kappa \to \kappa$ in $M$ and $j < \cf(\alpha)$.  To see $\gamma < \alpha$, consider the set
			$C = \{ \mu < \kappa  \mid f(z) \in \mu \text{ for all } z \in \bigcup_{\eta < \mu}[\eta]^j\}$. $C \in M$ since $f,j \in M$, and it is closed unounded by our $\GCH$ assumption.
			Since $\alpha$ belongs to every closed unbounded set in $M$ and $[\alpha]^j = \bigcup_{\eta < \alpha}[\eta]^j$, it follows that $\gamma < \alpha$. 
		\end{enumerate}
	\end{proof}
	
	\begin{remark}
		For the reader who is interested in obtaining similar results without cardinal arithmetic assumptions, we note that the $\GCH$ assumption in the last proof can be replaced with an approchability assumption concerning $\alpha$. Indeed, suppose that $\alpha$ is approchable with respect to a sequence $\vec{a} = \la a_i \mid i < \kappa\ra \subset [\kappa]^{<\kappa}$ which belongs to $M$. Namely, there exists a cofinal subset $x(\alpha) \subset \alpha$ of minimal ordertype, such that all proper initial segments of $x(\alpha)$ belong to $\vec{a}\uhr \alpha$. 
		Since $\vec{a} \in M$,  we can replace the function $\alpha \mapsto s(\alpha)$ with a similar function, in which the ladder system assignement $\alpha \mapsto \delta(\alpha)$ is replaced with $\alpha \mapsto x(\alpha)$. Then, assuming $\alpha$ is $\vec{a}$-approchable, we can replace the set $C$ in the proof above with the closed unbounded set $\{ \mu < \kappa  \mid f(a_i) \in \mu \text{ for all } i < \mu\}$. 
	\end{remark}
	
	\begin{definition}
		We say that an ideal $I$ on $\kappa$ is \emph{nonstationary} if it extends the nonstationary ideal on $\kappa$. Namely, $I$ contains every nonstationary set $Z \subset \kappa$. .
	\end{definition}
	
	
	Assuming $I$ is nonstationary, we can improve the result of Lemma \ref{Lem 2.2} by intersecting the set $B$ in the statement of the Lemma with the closed unbounded set $C^* = \bigcap\{ C \in M \mid C \subset \kappa \text{ is closed unbounded }\}$ which belongs to $\check{I}$. 
	
	\begin{corollary}\label{Cor 2.4}
		Let $M,I$, and  $A \in I^+ \cap M$ be as in the statement of Lemma \ref{Lem 2.2}. Suppose further that $I$ is nonstationary, then $A$ has an $M$-approximated $\lambda$-homogeneous subset $B \in I^+$ which is contained in every closed unbounded in $M$. 
	\end{corollary}

	The following Proposition summarises the results of this sub-section.
	\begin{proposition}\label{prop 2.5}
		Suppose $\mu < \kappa$ are regular cardinals and
		$\A$ is an algebra extending $\la H_\theta,\in,<_\theta\ra$ for a regular cardinal $\theta > 2^\kappa$. 
		Let $M \elem \A$ be a $\mu$-closed substructure of size $|M| = \mu$, and  
		$S \subset \kappa \cap \cof(\mu)$ be a stationary subset of $\kappa$ in $M$.
		If $S$ is positive with respect to a nonstationary, $\kappa$-complete
		, and $(\mu+1)-$closed ideal on $\kappa$, then for every regular cardinal $\lambda \in M \cap \kappa$, there exists a $\mu$-closed substructure $N \elem \A$ of size $|N| = \mu$, which is an end-extension of $M$ above $\lambda$, and satisfies $\sup(N \cap \kappa) \in S$.
	\end{proposition}
	
	

	\subsection{Closed nonstationary ideals from supercompactness assumptions}
	Following the argument that was given in the outset of this section, which described a construction of a desirable substructure $M$ as a union of a countable chain $M_n$ of elementary end-extensions, it is easy to see that Theorem \ref{thm 1.1} is an immediate consequence of  Proposition \ref{prop 2.5} and the following result.
	
	\begin{proposition}\label{prop 2.7}
		Suppose $\la \kappa_n \mid 1 \leq n < \omega\ra$ is an increasing sequence of supercompact cardinals in a model of set theory $V$. Then there exists a generic extension $V[G]$ of $V$ with the following properties:
		\begin{enumerate}
			\item  $\GCH$ holds below $\aleph_\omega$,
			\item $\kappa_n = \aleph_n^{V[G]}$ for every $1 \leq n < \omega$,
			\item  For every uncountable cardinal $\mu = \kappa_k$ for some $k < \omega$, an integer $n > k+1$ and a stationary set $S \subset \kappa_n$, there exists a nonstationary, $\kappa_n$-complete, and $(\mu+1)$-closed ideal $I$ on $\kappa_n$ such that $S \in I^+$.
		\end{enumerate}
	\end{proposition}

	The ideals $I$ which are used to prove the statement of Proposition \ref{prop 2.7} are obtained via the technique of generic elementary embeddings. 
	
	We commence by describing the basic construction of ideals from elementary embeddings, and their basic properties. 
	Let $j : V \to M$ be an elementary embedding of transitive classes, with critical point $\cp(j) = \kappa$, which satisfies ${}^\kappa M \subset M$. 
	Suppose that $\po \in V$ has the property that $j(\po)$ subsumes $\po$, and there exists a $\po$ name of a condition $g \in j(\po)/\po$ which is forced to extend $j(p)$ for every $p \in \G{\po}$ ($g$ is called a master condition for $j$, $\po$).  
	Working in a generic extension $V[\G{\po}]$, for every ordinal $\gamma$, $\kappa \leq \gamma < j(\kappa)$, and a condition $r \in j(\po)/\po$ which extends $g$, 
	we define an ideal $I_{\gamma,r}$ on $\kappa$ by 
	\[I_{\gamma,r} = \{ \name{X}_G \subset \kappa \mid r \force_{j(\po)/\po} \can{\gamma} \not\in j(\name{X})\}.\]
	%
	The fact $r$ extends $j(p)$ for every $p \in G$ guarantees that for every $X \subset \kappa$, the assertion $X \in I$ does not depend on a choice of a $\po$ name for $X$.
	The following well-known basic properties are immediate consequences of our definitions: (We refer the reader to \cite{For-HB} for a proof)
	\begin{fact}\label{fact 2.8}${}$
		\begin{enumerate}
			\item  $I_{\gamma,r}$ is a $\kappa$-complete and nonprincipal ideal on $\kappa$ ,
			%
			\item $I_{\gamma,r}$ is nonstationary if and only if $r \force \gamma \in j(\name{C})$ for every $\po$-name $\name{C}$ of a closed unbounded subset of $\kappa$,

			\item The forcing $j(\po)/\po$ absorbs the  forcing of $I_{\gamma,r}$-positive sets (which by our definition is equivalent as a forcing notation to the forcing $I_{\gamma,r}^+$ modulo $\equiv_{I_{\gamma,r}}$). 
			Moreover, if $j : V \to M$ is derived from an ultrapower by a $\kappa$ complete measure $U$, then the following describes a forcing projection from $j(\po)/\po$ to a dense subset of $I_{\gamma,r}^+$ (mod $\equiv_{I_{\gamma,r}}$): Let $B \in U$, and fix $f_r : B \to \po$ and $f_\gamma : B \to \kappa$ which represent $r$ and $\gamma$ in  $M$, respectively. For every condition $q \in j(\po)/\G{\po}$ extending $g$, and a function $f_q :B \to \po$ which represents $q$, we define $\pi(q)$ to be the $I_{\gamma,r}$ equivalence class of the set $\{ f_\gamma(x) \mid x \in B \text{ and } f_q(x) \in \G{\po} \}$. 
			Consequently, if $j(\po)/\po$ is $(\mu+1)$-closed for some $\mu < \kappa$, then $I_{\gamma,r}$ is a $(\mu+1)$-closed ideal.
		\end{enumerate}
	\end{fact}
	Let $\kappa < \eta$ be regular cardinals.
	Recall that $\kappa$ is $\eta$-supercompact if there exists an elementary embedding $j : V \to M$ with $\cp(j) = \kappa$, $\eta < j(\kappa)$, and 
	${}^\eta M \subset M$.
	For a regular cardinal $\rho < \kappa$, we denote the Levy collapse poset for collapsing the cardinals strictly between $\rho$ and $\kappa$ by $\Col(\rho,<\kappa)$. The conditions of $\col(\rho,<\kappa)$ are partial functions $h : \rho \times \kappa \to \kappa$ of size $|h| < \rho$, which satisfy $h(\alpha,\beta) < \beta$ for every $(\alpha,\beta) \in \dom(h)$. 
	We turn to prove the main Proposition.\\
	
	\noindent    \emph{Proof.} \textbf{(Proposition \ref{prop 2.7}).}
	Let $\la \kappa_n \mid 1\leq n < \omega\ra$ be a strictly increasing sequence of cardinals so that each $\kappa_n$ is $\kappa_\omega^+$-supercompact, where  $\kappa_\omega = \sup_{n<\omega}\kappa_n$. We set $\kappa_0 = \omega$ and define
	$\po = \la \po_n, \po(n) \mid n < \omega\ra$ to be the 
	full-support iteration where
	for every $n < \omega$,  $\po(n)$ is the $\po_n$ name of the Levy collapse poset $\Col(\kappa_n,<\kappa_{n+1})$.
	
	Let $\G{\po} \subset \po$ be a generic filter over $V$. It is easy to see $V[\G{\po}]$ satisfies $\GCH$ below $\aleph_\omega$ and $\kappa_n = \aleph_n^{V[\G{\po}]}$ for every $n < \omega$.
	For every $n \geq 1$, $\po$ naturally decomposes into three parts $\po = \po_{n-1} * \Col(\kappa_{n-1},<\kappa_n) * (\po/\po_n)$, where $\po_{n-1}$ satisfies $\kappa_{n-1}$.c.c and $\Col(\kappa_{n-1},<\kappa_n)* \po/\po_n$ is $\kappa_{n-1}$-closed. 
	Let $j : V \to M$ be a  $\kappa_\omega^+$ supercompact embedding with  critical point $\cp(j) =\kappa_n$ for some $n \geq 1$. 
	We note that $j(\po_{n-1}) = \po_{n-1}$ and $\Col(\kappa_{n-1},<\kappa_n) * (\po/\po_n)$ is a $\kappa_{n-1}$-closed poset of size $\kappa_{\omega}^+ < j(\kappa_n)$.
	The following well-known results follow from standard arguments concerning absorption of collapse posets, and supercompact embeddings (see \cite{Cum-HB} for a detailed account and proofs).
	\begin{itemize}
		\item $j(\po/\po_n)$ absorbs 
		$\Col(\kappa_{n-1},<\kappa_n) * (\po/\po_n)$. Moreover, there exists a projection of $j(\po/\po_n)$ to     $\Col(\kappa_{n-1},<\kappa_n) * (\po/\po_n)$, whose induced quotient $j(\po)/\po$ is $\kappa_{n-1}$-closed.
		\item $M[\G{\po}]$ contains a master condition $g$ for $j$ and $\po$, which is of the form 
		$g =  0_{\po_n} \fr \la p_k \mid n \leq k < \omega\ra$, where for each $k$, $p_k =\cup j``\G{\po(k)}$.
	\end{itemize}
	
	Let $\mu = \kappa_k$ be a regular cardinal below $\aleph_\omega^{V[\G{\po}]} = \kappa_\omega$, and $S \subset \kappa_n \cap \cof(\mu)$ be a stationary subset of $\kappa_n$ for some $n > k+1$. We claim there exists a nonstationary, $\kappa_n$-complete, and $(\mu+1)$-closed ideal $I$ on $\kappa_n$, such that $S \in \check{I}$.The ideal $I$ will be of the form $I_{\gamma,r}$ for a specific choice of $r$ and $\gamma$.
	Consider the decomposition of $\po$ to $\po = \po_{n+1} * \po/\po_{n+1}$. 
	Since the quotient $\po/\po_{n+1}$ is $\kappa_{n+1}-$closed, every subset of $\kappa_n$ in $V[\G{\po}]$ depends only on $\po_{n+1}$. Further, since $\po_{n+1}$ satisfies the $\kappa_{n+1}$.c.c, we can define an enumeration $\vec{C} = \la \name{C_i} \mid i < \kappa_{n+1}\ra$ in $V$, of all $\po_{n+1}$-names of closed unbounded subsets of $\kappa_n$. 
	Recall that $j : V \to M$ is a $\kappa_{\omega}^+$-supercompact embedding and therefore $j``\vec{C} = \la j(\name{C_i}) \mid i < \kappa_{n+1}\ra$ belongs to $M$, and is a sequence of length $\kappa_{n+1}$ of $j(\po)$-names for closed unbounded subsets of $j(\kappa_n) > \kappa_{n+1}$. Therefore, the empty (trivial) condition of $j(\po)$ forces that $C^* = \bigcap_{i < \kappa_{n+1}}j(\name{C_i})$ is also a closed unbounded subset of $j(\kappa_n)$. 
	Now, if $\name{S}$ is a $\po$-name for $S$, then there exists $p \in G$ so that $p \force \name{S} \text{ is stationary in } \can{\kappa_n}$. It follows that $g \geq j(p)$ forces that $``j(\name{S}) \text{ is stationary in } j(\can{\kappa_n})``$, and thus also that $``j(\name{S}) \cap C^* \text{ is unbounded in } j(\kappa_n)``$. 
	There must be therefore a condition  $r \in {j(\po)/\po}$ extending $g$, and an ordinal $\gamma \geq \kappa_n$ such that $r \force \can{\gamma} \in j(\name{S}) \cap C^*$. Let $I = I_{\gamma,r}$.
	By the basic facts listed above and our choice of $r,\gamma$, it is clear $I$ is nonstationary, $\kappa_n$-complete ideal on $\kappa_n$, and that $S \in \check{I}$. Moreover, 
	the fact $j(\po)/\po$ is $\kappa_{n-1}$-closed implies $I$ is $\kappa_{n-1}$-closed, and in particular $(\mu+1)$-closed (as $\mu = \aleph_k$ and $k + 1 < n$).
	\qed{Proposition \ref{prop 2.7}}\\
	\qed{Theorem \ref{thm 1.1}}

	\section{Mutual Stationarity in Two Cofinalities}\label{section-2}
	Building on the results of the previous section, we prove Theorem \ref{thm1.2} which improves the result of Theorem \ref{thm 1.1} to stationary sequences of two cofinalities where one of them is $\omega$. 
	To prove the theorem, we appeal to the method of \cite{ForMag-MS}, which makes use of certain Namba-type trees to prove $\MS(\vec{\kappa},\omega)$ in $\ZFC$. 
	By applying this method to trees whose splitting levels correspond to the ideals used to prove Theorem \ref{thm 1.1}, we show that in the model $V[\G{\po}]$ of Proposition \ref{prop 2.7}, if  $\vec{S} = \la S_n \mid k \leq n < \omega\ra$ is a stationary sequence with $S_n \subset \omega_n \cap (\cof(\omega) \cup \cof(\omega_k))$ then $\vec{S}$ is mutually stationary.
	
	\begin{proof}\textbf{(Theorem \ref{thm1.2})}        
		We work in the model $V[\G{\po}]$ from the proof of Proposition \ref{prop 2.7}. Let $k < \omega$ and $\vec{S} = \la S_n \mid k < n < \omega\ra$ be as in the statement of the Theorem . 
		By further shrinking the stationary sets $S_n$, we may assume that each $S_n$ is contained in either $\omega_n \cap \cof(\omega)$ or $\omega_n \cap \cof(\omega_k)$.
		Therefore, the sets $A_0 = \{ n < \omega \mid S_n \subset \omega_n \cap \cof(\omega)\}$ and 
		$A_k = \{ n < \omega \mid S_n \subset \omega_n \cap \cof(\omega_k)\}$ form a partition of $\omega\setminus (k+1)$. 
		If $A_0$ is finite, then the mutual stationarity of $\vec{S}$ is an immediate consequence of Theorem \ref{thm 1.1} (applied to a tail of $S_n$ with $S_n \subset \cof(\omega_k)$) and the remark in the preliminary section about finite modifications of stationary sequences. 
		Therefore, suppose $A_0$ is infinite and pick a function $\ell : A_0 \to A_0$ satisfying that $|\ell^{-1}(n)| = \aleph_0$ for every $n \in A_0$.
		Let $\A$ be an algebra on some $\la H_\theta,\in, <_\theta\ra$ for some regular $\theta > \aleph_\omega$, and let $M \elem \A$ be an $\omega_k$-closed substructure.
		
		\begin{definition}${}$
			\begin{enumerate}
				\item We say that a finite increasing sequence of ordinals $\vec{\eta} = \la \alpha_1,\dots, \alpha_n\ra$ is \textbf{valid} (with respect to $\vec{S}$ and $\ell$)
				if for each $m$, $1 \leq m \leq n$,  either $m \in A_k$ and $\alpha_m \in S_m$, or $m \in A_0$ and $\alpha_m \in (\kappa_{\ell(m)-1},\kappa_{\ell(m)})$. 
				
				\item Let $\delta$ be a ladder system on $\kappa_{\omega}$ and recall the induced function $s$ from Section \ref{section-1}, defined by 
				$s(\alpha) = \{ \delta(\alpha)\uhr j \mid j < \otp(\delta(\alpha))\}$. Define a modified function $s^*$ on valid ordinals as follows. If $m \in A_k$ and  $\alpha \in S_m$ then let $s^*(\alpha) = s(\alpha)$. Otherwise, $m \in A_0$ , $\alpha \in \kappa_{\ell(m)} \setminus \kappa_{\ell(m)-1}$ and define $s^*(\alpha) = \{ \alpha\}$.
				
				\item For every valid sequence $\vec{\eta} = \la \alpha_1,\dots,\alpha_n\ra$ we define $s^*(\vec{\eta}) = \cup_{1 \leq m \leq n}s^*(\alpha_m)$ and $M(\vec{\eta}) =  \SK^{\A}(M \cup s^*(\vec{\eta}))$. A straightforward modification of the proof of Lemma \ref{Lem 2.3} shows $M(\vec{\eta})$ is $\mu$-closed.
			\end{enumerate}
		\end{definition}
		
		Let $T \subset [\kappa_{\omega}]^{<\omega}$ be the tree of all valid sequences. 
		Propositions \ref{prop 2.5} and \ref{prop 2.7} guarantee $T$ has a nonempty subtree $T_k$ with a stem $t_{k+1}$ of length $k+1$, such that the following holds for every $\vec{\eta} \in T_k$:
		\begin{itemize}
			\item If $n = |\vec{\eta}|$ belongs to $A_k$ then $\suc_{T_0}(\vec{\eta})$ consists of a unique ordinal $\alpha_{\vec{\eta}} \in S_n$ which belongs to a  $\kappa_{n-1}$-homogeneous and approximated set in $M(\vec{\eta})$, and is further contained in every closed unbounded subset of $\kappa_n$ in $M(\vec{\eta})$. 
			Hence, $M(\vec{\eta} \fr \la \alpha_{\vec{\eta}}\ra)$ is an elementary end extension of $M(\vec{\eta})$ above $\kappa_{n-1}$, and meets $S_n$. 
			\item If $n = |\vec{\eta}|$ belongs to $A_0$ then $\suc_{T_0}(\vec{\eta})$ is an unbounded subset of $\kappa_{\ell(n)}$ which is $\kappa_{\ell(n)-1}$-homogeneous and approximated in $M(\vec{\eta})$. In particular, for every $\alpha \in \suc_{T_0}(\vec{\eta})$, $M(\vec{\eta} \fr \la \alpha\ra)$ is an elementary end extension of $M(\vec{\eta})$ above $\kappa_{\ell(n)-1}$. 
		\end{itemize}
		It clearly follows that for every maximal branch $b$ in $T_0$, the model $M(b) = \cup_{m<\omega}M(b\uhr m)$ satisfies $\sup(M(b) \cap \kappa_n ) \in S_n$ for every $n \in A_k$. 
		We proceed to define a $\subset$-decreasing sequence  $\la T_m \mid k+1 <  m < \omega\ra$ of subtrees of $T_k$, which satisfy the following conditions:
		\begin{enumerate}
			\item The length of the stem $t_m$ of $T_m$ is at least $m$,
			\item for every $n \in A_0 \setminus m$ and $\vec{\eta} \in T_m$ of length $|\vec{\eta}| \in \ell^{-1}(n)$,  $\suc_{T_m}(\vec{\eta})$ is an unbounded subset of $\kappa_n$, 
			\item for every $n \in A_0 \cap m$ and $\vec{\eta} \in T_m$ of length $|\vec{\eta}| \in \ell^{-1}(n)$,  $\suc_{T_m}(\vec{\eta}) = \{ \alpha_{\vec{\eta}}\}$ is a singleton,
			\item 
			for every $n \in A_0 \cap m$ there exists an ordinal $\delta_n \in S_n$ such that $\sup(M(b) \cap \kappa_n) = \delta_n$ for every maximal branch $b$ in $T_m$.
		\end{enumerate}
		A sequence of trees with these properties have a common maximal branch $b$, for which $M(b) \elem \A$ meets every $S_n$.
		Let us describe the inductive step of the construction $T_m \mapsto T_{m+1}$. Suppose  $T_m$ has been constructed. If $m \in A_k$ then nothing needs to be done, since by the definition of $T_k$, the length of the stem of $T_m$ which is guaranteed to be at least $m$, must actually have length $m+1$. 
		
		Suppose $m \in A_0$. For every $\delta \in \kappa_{m}$ consider the cut-and-choose type game $\game_\delta$, played by two players G (good) and B (bad) on $T_m$, in which they build an increasing sequence of nodes $\vec{\eta}_0,\vec{\eta}_1,\dots$ in $T_m$, resulting in a cofinal branch $b$ in $T_m$. At  round $r < \omega$ of the game, given $\vec{\eta}_r \in T_m$, player B chooses an ordinal $\beta_r <  \sup(\suc_{T_m}(\vec{\eta}_r))$, and if $\ell(r) = m$ then $\beta_r$ needs to be below $\delta$. Then, player G is required to choose an ordinal $\alpha_r \in \suc_{T_m}(\vec{\eta}_r) \setminus \beta_r$, which determines the next node - $\vec{\eta}_{r+1} = \vec{\eta}_r \fr \la \alpha_r\ra$. Note that if $\suc_{T_m}(\vec{\eta}_r)$ is a singleton $\{ \alpha\}$, then B must choose an ordinal $\beta < \alpha$, and G has to pick $\alpha_r = \alpha$. At any stage of game, any player who fails to play according to these requirements looses. If the game continuous for infinitely many rounds and produces a branch $b$, then G wins if and only if $\sup(M(b) \cap \kappa_m) \leq \delta$. 
		Since any violation of $\sup(M(b) \cap \kappa_m) \leq \delta$ is achieved sat a finite stage $r$ via $M(\vec{\eta_r})$, the payoff set for player B is open, and thus $\game_\delta$ is determined. 
		
		For each $\delta < \kappa_m$, let $\sigma_\delta$ be a winning strategy (for either G or B) in $\game_\delta$. 
		Pick a regular cardinal $\theta^*$ above $\theta$ and let $\A^* = ( H_{\theta^*}, \in, \A,M,\vec{S},T_m,\la \sigma_\delta \mid \delta < \kappa_m\ra)$. 
		As shown in \cite{ForMag-MS}, if $\delta < \kappa_m$ satisfies that $\SK^{A^*}(\delta) \cap \kappa = \delta$ then $\sigma_\delta$ cannot be a winning strategy for B. The idea is that if $\sigma_\delta$ were a strategy for B, then we can form a play of $\sigma_\delta$  whose moves $\vec{\eta_0},\vec{\eta_1}, \dots$ are all in $\SK^{A^*}(\delta)$. This contradicts the assumption of B winning, as the resulting structure $M(b)$ must be contained in $\SK^{A^*}(\delta)$. 
		It follows that $\sigma_\delta$ is winning for G, for a closed unbounded set of $\delta < \kappa_m$. In particular, it contains an ordinal $\delta \in S_m \setminus \sup(M \cap \kappa_m)$. Fix such an ordinal $\delta$, and pick a cofinal sequence $\la \delta_p \mid p < \omega\ra$ in $\delta$. 
		We shall use this sequence 
		and $\sigma_\delta$ to construct $T_{m+1}$ level by level. 
		At each stage, we make sure the sequences $\vec{\eta}$ which are added to $T_{m+1}$ correspond to legal plays of $\sigma_\delta$. 
		It is clear that the first $m$ rounds of any play with $\sigma_\delta$ leads to $\vec{\eta}_m = t_m$, the stem of $T_m$. 
		We then play one more round to pick an ordinal $\alpha_{m}$ from $\suc_{T_{m}}(t_m)$ and set $\suc_{T_{m+1}}(t_m) = \{ \alpha_m\}$. 
		This determines the first $(m+1)$ levels of $T_{m+1}$. We proceed by induction to define the $n$-th level of $T_{m+1}$ for every $n > m+1$. Let $\vec{\eta}$ be a node on the $n$-th level of $T_{m+1}$. If $\suc_{T_m}(\vec{\eta}) = \{\alpha\}$ is a singleton then we define 
		$\suc_{T_{m+1}}(\vec{\eta}) = \{\alpha\}$. Otherwise, $\suc_{T_m}(\vec{\eta}) \subset \kappa_{\ell(n)}$ is unbounded, and we consider the following two cases:
		\begin{enumerate}
			\item Suppose $\ell(n) = m$. Note that  $\suc_{T_m}(\vec{\eta}) \cap \delta$ must be unbounded in $\delta$ since otherwise B could play $\beta_n = \sup(\suc_{T_m}(\vec{\eta}) \cap \delta) + 1$ leaving  $\sigma_\delta$ without legal moves. If $n$ is the $p$-th element of $\ell^{-1}(m)$ (recall $\ell^{-1}(m)$ is infinite) then let $\alpha_{\vec{\eta}} < \delta$ be the response of $\sigma_\delta$ to B playing $\beta_n = \delta_p$.
			Clearly, $\delta_p \leq \alpha_{\vec{\eta}} < \delta$.
			
			\item Suppose that $\ell(n) \neq m$, then $\ell(n) > m$ since $\suc_{T_m}(\vec{\eta})$ is not a singleton. We define an increasing sequence $\la \alpha_{\vec{\eta}}(i) \mid i < \kappa_{\ell(n)}\ra$ of ordinals in 
			$\suc_{T_m}(\vec{\eta})$ by induction on $i < \kappa_{\ell(n)}$.
			Define $\alpha_{\vec{\eta}}(0)$ to be the ordinal response of $\sigma_\delta$ to B playing $\beta_n = 0$. Suppose now $\alpha_{\vec{\eta}}(i)$ has been defined for all $i$ below some $i^* < \kappa_{\ell(n)}$. Define $\beta _{\vec{\eta}}(i^*) = \bigcup_{i<i^*}\alpha_{\vec{\eta}}(i)$ and let $\alpha_{\vec{\eta}}(i^*)$ be the ordinal response of $\sigma_\delta$ to B playing $\beta_n = \beta _{\vec{\eta}}(i^*)+1$. 
			We finally set $\suc_{T_{m+1}}(\vec{\eta}) = \{\alpha_{\vec{\eta}}(i) \mid i < \kappa_{\ell(n)}\}$. 
		\end{enumerate}
		This concludes the construction of $T_{m+1}$. It is straightforward to verify $T_{m+1}$ satisfies the first three conditions listed above. Let us verify $T_{m+1}$ satisfies the fourth condition for $n = m \in A_0 \cap (m+1)$.
		Suppose that $b \subset T_{m+1}$ is a cofinal branch. On the one hand, $b$ is a result of a $\sigma_\delta$ play, and therefore $\sup(M(b) \cap \kappa_{m}) \leq \delta$. On the other hand, our choices of $\suc_{T_{m+1}}(\vec{\eta})$ for every $\vec{\eta} \in T_{m+1}$ with $|\vec{\eta}| = n \in \ell^{-1}(m)$, guarantee  $\sup(M(b) \cap \kappa_{\ell(n)})$ is above $\delta_p$ for every $p < \omega$. We conclude that  
		$\sup(M(b) \cap \kappa_{\ell(n)}) = \delta$ for every maximal branch $b$ of $T_{m+1}$. 
	\end{proof}

	\section{Mutual Stationarity and Prikry-type forcing}\label{section-3}
	
	We present a variant of the ideal-based construction of homogeneous and approximated sets from Section \ref{section-1}. The ideals we will be using naturally emerge from Prikry-type forcing notions; hence their name - ``Prikry-closed ideals``. 
	
	\begin{definition}
		Let $I$ be a $\kappa$-complete non-principal ideal on a cardinal $\kappa$, and let $\mu < \kappa$ be a regular cardinal. 
		We say that $I$ is $\mu$-Prikry-closed if there exists a dense set $D$ of $I^+$ and a $\mu$-closed suborder $\leq^*_I$ of $\leq_I\uhr D$ which satisfies the following two conditions:
		\begin{itemize}
			\item $\leq^*_I$ respects $\equiv_I$ equivalency. Namely, if $B^* \leq^*_I B$ and $B' \equiv_I B^*$ then $B' \leq^*_I B$. 
			\item  For every $B \in D$ and $f : B \to 2$ there exists a set $B^* \leq_I^* B$ such that $f\uhr B^*$ is constant.
		\end{itemize}
	\end{definition}
	
	It is easy to see that if $\leq^*_I$ is $(\lambda+1)$-closed then we can replace the function $f : B \to 2$ in the second condition of the definition with a function $f : B \to \lambda$. 
	Consequently, by repeating the arguments of Lemma \ref{Lem 2.2}, replacing the $(|M|+1)$-closure assumption of $I$ with a $(\max\{|M|,\lambda\}+1)$-Prikry-closure assumption, we obtain the following analogous result. 
	
	\begin{lemma}\label{Lem 4.1}
		Let $M$ be a $|M|$-closed substructure of an algebra $\A$. If $\lambda < \kappa$ are regular cardinals in $M$ and $I \in M$ is a $\kappa$-complete and $(\max\{|M|,\lambda\}+1)$-Prikry-closed ideal on $\kappa$,
		then every $A \in I^+ \cap M$ has a subset $B \in I^+$ of $A$ which is $\lambda$-homogeneous and approximated in $M$.
	\end{lemma}
	
	\subsection{Constructing Prikry-closed ideals on accessible cardinals}
	Starting from an increasing sequence of measurable cardinals $\la \kappa_n \mid n < \omega\ra$ such that each $\kappa_n$ has a $(*,\kappa_{n-1}^+)$ sequence\footnote{For $n = 0$, we replace $\kappa_{-1}^+$ with $\omega$.}, we show there exists a generic extension in which $\kappa_n = \aleph_{2n+1}$ for each $n < \omega$ and that every stationary set $S \subset \kappa_n \cap \cof(<\kappa_{n-1}^+)$ is positive with respect to a $\kappa_{n-1}$-Prikry-closed nonstationary ideal $I$ on $\kappa_n$.
	By Lemma \ref{Lem 4.1}, this generic extension satisfies the result of Theorem \ref{thm1.3}.
	
	To obtain the desired forcing extension, we construct an $\omega$-iteration of posets $\qo(n)$ which satisfy the following conditions:
	\begin{enumerate}
		\item $|\qo(n)| = \kappa_n^+$,
		\item $\qo(n)$ collapses all the cardinals in the interval $(\kappa_{n-1}^+,\kappa_n)$,
		\item $\qo(n)$ is $\kappa_{n-1}^+$-distributive and is subsumed by a Prikry-type forcing $\qo(n)^*$ whose direct extension order is $\kappa_{n-1}^+$-closed.
	\end{enumerate} 
	
	By using Gitik's method of iterating distributive forcings which embed into Prikry-type posets, it is possible to 
	form a Prikry-type iteration $\bar{\qo} = \la \bar{\qo}_n, \qo(n) \mid n < \omega\ra$ of the posets $\qo(n)$, so that for every $n < \omega$, the tail $\bar{\qo}/\qo_n$ of the iteration  $\bar{\qo}$ does not add new bounded subsets to  $\kappa_{n-1}^+$.
	Therefore, for all purposes involving the subsets of $\kappa_n$ and ideals on $\kappa_n$ in a $\bar{\qo}$ forcing extension, it is sufficient to consider its intermediate extension by $\qo_{n+1} \cong \qo_n *  \qo(n)$. Since the forcing $\qo_n$ has size $\kappa_{n-1}^+$, which is small with respect to $\kappa_n$, the  $(*,\kappa_{n-1}^+)$ assumption in $V$ still holds in a $\qo_n$ generic extension $V[\G{\qo_n}]$.
	It is therefore sufficient to focus on a single ``local`` step of the construction and prove the following statement.
	Fix $n < \omega$, and denote $\kappa_{n-1}^+$ by $\lambda$ and $\kappa_n$ by $\kappa$. 
	\begin{proposition}\label{prop4.1}
		Suppose $\lambda < \kappa$ are regular cardinals in a ground model $V$ so that $\kappa$ carries a $(*,\lambda)$ sequence. Then there exists a forcing notion $\qo$ which satisfies the following conditions:
		\begin{enumerate}
			\item $|\qo| = \kappa^+$,
			\item $\qo$ collapses all the cardinals in the interval $(\lambda,\kappa)$,
			\item $\qo$ is $\lambda$-distributive and is subsumed by a Prikry-type forcing $\qo^*$ whose direct extension order is $\lambda$-closed,
			\item every stationary set $S \subset \kappa \cap \cof(<\lambda)$ in a $\qo$ generic extension is positive with respect to a
			$\lambda$-Prikry-closed nonstationary ideal on $\kappa$.
		\end{enumerate}
	\end{proposition}
	
	In the proof of Proposition \ref{prop4.1} we rely on a method of Gitik from \cite{Gitik-iteration}, showing that from assumptions similar to a $(*,\lambda)$-sequence, there exists a forcing extension $V[\G{\qo}]$ in which $\kappa = \lambda^+$ and the nonstationary ideal on $\kappa$ restricted to $\cof(<\lambda)$ is precipitous. 
	The idea is that after a certain preparatory forcing, it is possible to add closed unbounded sets to $\kappa$ which destroy the stationarity of all subsets $S \subset \kappa \cap \cof(<\lambda)$ which are not positive with respect to a certain natural filter-extension of a normal measure $U_{\kappa,\tau}$ on $\kappa$ in the ground model. Therefore, the forcing $\qo = (\po*\col) * \CC$ consists of two main parts: The preparatory forcing $\po*\Col$ which adds many supercompact Prikry and Magidor sequences to cardinals $\nu < \kappa$ and which collapses all the cardinals in the interval $(\lambda,\kappa)$, and a club shooting iteration $\CC$ which is responsible for destroying the stationarity of all ``bad`` sets $S$.
	
	We proceed to describe the forcing $\qo$ and highlight its key properties needed for the proof of Proposition \ref{prop4.1}. We assume the reader is familiar with the basic terminology and results concerning supercompact and measurable cardinals, and Prikry-type posets.
	
	\noindent \textbf{On the ground model assumptions and the preparation forcing -} 
	Our assumption, that $\kappa$ carries a $(*,\lambda)$-sequence stipulates the existence of a coherent sequence 
	$\W = \la W_{\nu,\tau} \mid \nu \leq \kappa, \tau < o(\nu)\ra$ of supercompact measures with the following properties: 
	\begin{itemize}
		\item For each $\nu \leq \kappa$ with $o(\nu) > 0$, $W_{\nu,\tau}$ is a $\nu^+$-supercompact measure on $\nu$. Namely, it is a $\nu$-complete fine normal measure on $\power_\nu(\nu^+)$. In particular, by standard arguments, each $W_{\nu,\tau}$ concentrates on the set of $x \in \power_\nu(\nu^+)$ with $x \cap \nu \in \nu$ and $\otp(x) = (x \cap \nu)^+$. 
		We refer to a set $x$ with these properties as a \textbf{supercompact point}, and denote $x \cap \nu$ by $\nu_x$. 
		\item $o(\kappa) = \lambda$, and $o(\nu) < \lambda$ for all $\nu < \kappa$,
		\item For each $W_{\nu,\tau}$, let $U_{\nu,\tau}$ be its projection to a normal measure on $\nu$, defined by $A \in U_{\nu,\tau}$ if and only if $\{ x \in \power_\nu(\nu^+)  \mid x \cap \nu \in A\} \in W_{\nu,\tau}$. The sequence $\U = \la U_{\nu,\tau} \mid \nu \leq \kappa, \tau < o(\nu)\ra$ is a coherent sequence of normal measures.
	\end{itemize}
	For each $\tau < \lambda$, let  $j_\tau : V \to M_\tau \cong \Ult(V,W_{\kappa,\tau})$ denote the ultrapower embedding of $V$ by $W_{\kappa,\tau}$.
	
	The first part of the preparatory forcing is an iteration of Prikry-type posets,  $\po = \la \po_\nu,\po(\nu) \mid \lambda < \nu < \kappa\ra$, which add a cofinal supercompact Prikry/Magidor sequence $\x_\nu \subset \power_\nu(\nu^+)$ to each $\nu < \kappa$ with $o(\nu)>0$. 
	Since we wish to start our iteration above $\lambda$, we set $\po_{\lambda+1}$ to be the trivial forcing, and
	for every ordinal $\nu$,  $\lambda < \nu < \kappa$, $\po(\nu)$ is either the trivial forcing or a Prikry-type forcing of size $|\po(\nu)| \leq 2^{\nu^+}$, whose direct extension order is $\nu$-closed. 
	The manner in which the posets $\po(\nu)$ are iterated is called the Gitik iteration. Conditions $q \in \po$ have the form $q = \la q(\gamma) \mid \gamma \in \supp(q)\ra$ where $\supp(q)$ is an Easton subset of $\kappa$\footnote{namely, $\supp(p)\cap \delta$ is bounded in $\delta$, for every $\delta \leq \kappa$ regular.}, and for every $\gamma \in \supp(q)$,  $q\uhr \gamma$ belongs to $\po_\gamma$ and forces $q(\gamma) \in \po(\gamma)$. When extending $q$ in $\po$, only finitely many non-direct extensions of $q(\gamma)$, $\gamma \in \supp(q)$, are allowed. A condition $q^* \in \po$ is a direct extension of $q$ if $q^*(\gamma)$ is a direct extension of $q(\gamma)$ for all $\gamma \in \supp(q)$.  In \cite{Gitik-iteration}, it is shown that a Gitik iteration as above has the following properties.\\

	\begin{fact}[Basic properties of $\po$]${}$
		
		\begin{enumerate}
			\item For all $\nu \leq \kappa$, $\po_\nu$ and $\po/\po_{\nu}$ are of Prikry-type. 
			\item The direct extension order $\po/\po_{\nu}$ is $\nu$-closed. In particular  $\po/\po_{\nu}$ does not add new bounded subsets to $\nu$, and $\po \cong \po/\po_{\lambda+1}$ does not add new subsets to $\lambda$.
			\item $|\po_{\nu}| = 2^\nu$ and if $\nu$ is Mahlo then $\po_{\nu}$ satisfies $\nu$.c.c.
			\seti
		\end{enumerate}
	\end{fact}
	
	\noindent    Suppose $\po_{\nu}$ has been defined for some $\nu < \kappa$. 
	If $o(\nu) = 0$, $\po(\nu)$ is taken to be the trivial forcing.
	Otherwise, $\po(\nu)$ is a tree forcing which threads a supercompact Prikry/Magidor cofinal sequence $\x_\nu$ to $\power_\nu(\nu^+)$.
	As opposed to the original Magidor poset, which adds a closed unbounded sequence as well as its new initial segments at the same time, here, $\po(\nu)$ does not introduce new bounded subsets of $\nu$ to $V[\G{\po_\nu}]$, but instead generates $\x_\nu$ by generically threading smaller generic sequences $\x_\mu$ for $\mu < \nu$.
	Therefore given a generic assignment $\mu \to \x_\mu$ which is assumed to be derived from $\G{\po_\nu}$, 
	$\po(\nu)$ consists of pairs $p = \la t,T\ra$ which satisfy the following conditions.
	\begin{itemize}
		\item $t = \la x_0,\dots,x_{n-1}\ra$ is a $\subset$-increasing sequence of supercompact points in  $\power_\nu(\nu^+)$. $t$ determines an initial segment $\x_t$ of $\x_\nu$, 
		defined by $\x_t = \x_{x_0} \cup \{x_0\} \cup \x_{x_1} \cup \{x_1\} \cup \dots \cup 
		\x_{x_{n-1}} \cup \{x_{n-1}\}$, where for each $i \leq n-1$, $\x_{x_i} \subset x_i$ is the order-isomorphic copy of  $\x_{\nu_{x_i}} \subset \nu_{x_i}^+$, obtained via the inverse of the transitive collapse of $x_i$ onto $(\nu_{x_i})^+$. 
		\item  $T \subset [\power_{\nu}(\nu^+)]^{<\omega}$ is a tree which contains the possible options to extend $t$ (and thus $\x_t$). For each $s \in T$, the splitting levels of $T$ are required to be measure one with respect to $o(\nu)$ many measures in $V[\G{\po_\nu}]$, $\{W_{\nu,\tau}(t \fr s) \mid \tau < o(\nu)\}$,  which extend the
		ground model normal measures $\{W_{\nu,\tau} \mid \tau < o(\nu)\}$, and concentrate on the set of supercompact points $x \in \power_\nu(\nu^+)$ for which $\x_\nu$ is compatible with $\x_{t \fr s}$
		(see \cite{Gitik-iteration} for details on the extensions $W_{\nu,\tau}(t \fr s)$ of $W_{\nu,\tau}$).
	\end{itemize}
	
	When extending a condition $p$ we are allowed to shrink the tree $T$ (a direct extension, $\leq^*$) or add new points from $T$ to the sequence $t$ (a pure non-direct extension). In both cases, we are also allowed to switch $t$ with another sequence $t'$ for which $\x_t = \x_{t'}$.
	\\
	
	For every $\alpha \leq \kappa$ and $\tau^* \leq o(\alpha)$ the poset $\po(\alpha)$ has a natural variant denoted $\po(\alpha,\tau^*)$ in which the tree splitting levels correspond to the shorter list of  $\W$ measures $W_{\alpha,\tau}$ with $\tau < \tau^*$. In particular, we have that $\po(\alpha,o(\alpha)) = \po(\alpha)$, and that for every $\tau^* < o(\alpha)$, $\po(\alpha,\tau^*)$ coincides with $j_{\tau^*}(\po)(\alpha)$\footnote{i.e., stage $\alpha$ of the iteration $j_{\tau^*}(\po)$} as constructed in the $W_{\alpha,\tau^*}$ ultrapower model, $M_{\tau^*}$.
	Both $\po(\alpha)$ and its variants $\po(\alpha,\tau^*)$, are of Prikry-type with $\alpha$-closed direct extension orders. 
	
	Following \cite{Gitik-iteration}, it is routine to verify $\po$ satisfies the following additional properties.\\
	
	\begin{fact}[Basic properties of $\po$, continued]${}$
		\begin{enumerate}
			\conti
			\item For every $\nu < \kappa$, $\po$ changes its cofinality if and only if $o(\nu) >0$ and then $\cf(\nu)^{V[\G{\po}]} = \cf^V(\omega^{o(\nu)})$. In particular, if $o(\nu) = \rho$ is a regular cardinal then $\cf^{V[\G{\po}]}(\nu) = \rho$.
			\item For every $E \in \bigcap_{1 \leq \tau < \lambda}U_{\kappa,\tau}$ (in $V$), the set $E^+ = E \cup \cof(\geq \lambda)$ becomes $(\lambda+1)$-fat stationary in $V[\G{\po}]$. Namely, for every closed unbounded subset $C$ of $\kappa$, $C \cap E^+$ contains a closed set of order type $\lambda+1$. 
		\end{enumerate}
	\end{fact}
	
	\noindent The iteration $\po$ naturally projects to a similar iteration $\po^{\U} = \la \po^{\U}_\nu,\po^{\U}(\nu) \mid \nu < \kappa\ra$ which threads Prikry/Magidor sequences $c_\nu$, $\nu < \kappa$ or ordinals (i.e., as opposed to the $\x_\nu$ sequences of supercompact points). For each $\nu < \kappa$, we can derive the sequence $c_\nu$ from $\x_{\nu}$ be replacing each supercompact point $x \in \x_{\nu}$ with its ordinal projection $\nu_x = x \cap \nu$. 
	
	The second part of the preparatory forcing is a Levy collapse poset, $\Col = \Col(\lambda,<\kappa)$. 
	It follows from basic properies of $\po$, that for every $E \in \bigcap_{\tau < \lambda}U_{\kappa,\tau}$, the set $\bar{E} = E \cup \cof(\geq \lambda)$ remains $(\lambda+1)$-fat stationary in $V[\G{\po}*\G{\col}]$, and thus by a well-known argument of Abraham and Shelah (\cite{Abr-She-fatstat}), the poset $\CC[\bar{E}]$ of all closed bounded subsets $d \subset E$, ordered by end-extension, is $\kappa$-distributive. 
	
	For $\nu < \kappa$, we denote $\col(\lambda,<\nu)$ by $\col_\nu$.
	The fact that $\po/\po_{\nu}$ does not add new bounded subsets to 
	$\nu$ implies $\Col_\nu = \Col(\lambda,<\nu)^{V[\G{\po}]}$ belongs to the intermediate extension $V[\G{\po_{\nu}}]$. Moreover, assuming $\nu$ is regular in $V$,  $\Col(\lambda,<\nu)$ is $\nu$.c.c in $V[\G{\po_{\nu}}]$ and so, the genericity of a filter $G_\nu \subset \col_{\nu}$ is determined by its bounded pieces $G_\nu \uhr \mu$, $\mu < \nu$. Therefore, even though $\nu$ may become singular at stage $\nu$ of the iteration $\po$, the natural restriction of a $\col$ generic filter $\G{\col}$ over $V[G(\po)]$ to its bounded pieces must form a  $\col_{\nu}$ generic filter over $V[\G{\po_{\nu}}]$. Hence, $\po_{\nu} * \Col_\nu$ is subsumed by $\po * \Col$. \\
	
	\noindent\textbf{On the club shooting iteration - }
	The main part of our forcing $\qo$ is a club shooting iteration $\CC = \CC_{\kappa^+} = \la \CC_\eta, \CC(\eta) \mid \eta < \kappa^+\ra$. For every $\eta \leq \kappa^+$, the support of a condition $q = \la q_\gamma \mid \gamma \in \supp(q)\ra$ in $\CC_\eta$ is of size $|\supp(q)| < \kappa$, and for every $\gamma \in \supp(q)$, 
	$q(\gamma)$ is a $\CC_{\gamma}$-name of a closed bounded subset of $\kappa$ subject to certain additional conditions. By a standard argument, $\CC$ satisfies $\kappa^+$.c.c.
	The purpose of the iteration $\CC$ is to destroy the stationarity of the subsets of $\kappa \cap \cof(<\lambda)$ which fail to be positive with respect to a natural filter extension of $U_{\kappa,\tau}$ for some $\tau$, $1 \leq \tau < \lambda$.
	This result is not obtained by directly adding closed unbounded sets through the complements of all ``bad`` stationary sets, but rather adding certain $<\lambda$-clubs through the $V$ sets in $\bigcap_{1\leq \tau < \lambda} U_{\kappa,\tau}$.
	Let $\la E_\eta \mid \eta < \kappa^+\ra$ be an enumeration of all sets $E \in \bigcap_{1 \leq \tau < \lambda}U_{\kappa,\tau}$, and for every $\eta < \kappa^+$, let $\bar{E_\eta} = E_\eta \cup \cof(\lambda)$. 
	Gitik has shown (\cite{Gitik-precaleph2}) that a desirable extension -  in which all $\cof(<\lambda)$ stationary sets in $\kappa$ are positive with respect to a suitble extension of $U_{\kappa,\tau}$ - can be obtained by iteratively adding a sequence of closed unbounded sets $\vec{C} = \la C_\eta \mid \eta < \kappa^+\ra$ so that for every $\eta < \kappa^+$, $C_\eta$ is a subset of both $\bar{E_\eta}$ and the set of all $\vec{C}\uhr \eta$-generic points. An ordinal $\nu < \kappa$ is a $\vec{C}\uhr \eta$-generic if the sequence $\vec{C}\uhr\la \eta, \nu\ra = \la C_\xi \cap \nu \mid \xi \in \tau_\eta``\nu\ra$ is generic over $V[\G{\po_{\nu}}* \G{\Col_\nu})]$, where $\tau_\eta : \kappa \to \kappa$ is the $\eta$-th function in some fixed canonical sequence.\\
	
	The crux of the argument lies in the proof that the iteration $\CC$ is $\kappa$-distributive. Since $\CC$ is $\kappa^+$.c.c, this  amounts to showing its initial segments,  $\CC_\eta$, $\eta < \kappa^+$, are $\kappa$-distributive. 
	The last is proved by induction on $\eta < \kappa^+$, and is based on the fact that each $\CC_\eta$ is subsumed by the Priky type poset $\po(\kappa,\tau)$, for every $1 \leq \tau < \lambda$. 
	Once this is established for $\CC_{\eta}$,  $\po(\kappa,\tau)=j_{\tau}(\po)(\kappa)$, we can reflect this statement on a set which belongs to $U_{\kappa,\tau}$, for all $1 \leq \tau < \lambda$, and conclude that every set
	$E \in \bigcap_{1 \leq \tau < \lambda} U_{\kappa,\tau}$ contains a subset $E' \in \bigcap_{1 \leq \tau < \lambda} U_{\kappa,\tau}$ whose ordinals are
	potential generic ordinals for $\CC_\eta$. These are 
	ordinals $\nu < \kappa$ with the property that
	every condition $\vec{d} \in \CC_\eta\uhr \nu$ in $V[\G{\po}*\G{\col}]$ extends to a condition $\bar{d} \in \CC_\eta$ which forms a $\CC_\eta \uhr \nu$ generic filter over $V[\G{\po_\nu}*\G{\col_\nu}]$.  
	Then, with a fat stationary set $E' \subset E_\eta$ of potentially generic points for $\CC_{\eta}$, a straightforward extension of the Abraham-Shelah argument from \cite{Abr-She-fatstat} shows that $\CC_{\eta+1}$ is $\kappa$-distributive. The argument for limit stages is similar (see \cite{Gitik-iteration}).
	
	
	We note that the forcing $\po(\kappa,\tau)$ is defined in $V[\G{\po}]$ and is independent of the $\col$ generic filter $\G{\col}$. Therefore, when including the preparation forcing $\po*\col$, the entire poset can be described as $\po * (\col \times \po(\kappa,\tau)) \cong \po * (\po(\kappa,\tau) \times \col)$. \\
	
	\noindent\textbf{On the absorption argument and projection properties- }
	Let     $V[\G{\po}*\G{\col}]$ be a $V$-generic extension by $\po*\col$.
	We describe the main idea of Gitik's argument that $\CC_\eta$ is subsumbed by  $\po(\kappa,\tau)$ in $V[G(\po)* G(\Col)]$ (or equivalently, in $M_{\tau}[G(\po)* G(\Col)]$), for all $1 \leq \tau < \lambda$. For further details about the argument, we refer to \cite{Gitik-iteration}.
	
	To start, recall  that for a given $E_0\in \bigcap_{1 \leq \tau < \lambda}U_{\kappa,\tau}$, the set $\bar{E_0} = E_0 \cup \cof(\lambda)$  is fat stationary in $V[\G{\po}*\G{\col}]$ and therefore $\CC[\bar{E_0}]$ is $\kappa$-distributive in $V[\G{\po}*\G{\col}]$.
	Let $\la D_i \mid i < \kappa^+\ra$ be an enumeration of all dense open sets of $\CC[\bar{E_0}]$. 
	The forcing $\po(\kappa,1) = j_1(\po)(\kappa)$ adds a supercompact sequence  $\x_{\kappa} = \la x_n \mid n < \omega\ra$ of Prikry points in $\power_\kappa(\kappa^+)$. For each $n <\omega$, the fact $|x_n| < \kappa$ implies that the set $D_{x_n} = \cap_{i \in x_n}D_i$ is dense, and thus for each $d \in \CC[\bar{E_0}]$, it is possible to construct an $\omega$-increasing sequence $\la d_n \mid n < \omega\ra$ of conditions which extend $d$, so that
	$d_n \in D_{x_n}$ for each $n < \omega$. The sequence generates a $\CC[\bar{E_0}]$-generic filter over $V[\G{\po}* \G{\col}]$ since $\cup_n \x_n = \kappa^+$. We conclude $\po(\kappa,1)$ absorbs $\CC[\bar{E_0}]$. 
	Back in $V$, we can reflect this result on some subset $E_0'(1)\in U_{\kappa,1}$ of $E_0$ consisting of potentially generic ordinals.
	Moving one step up to $\po({\kappa,2})$, recall that every $\po(\kappa,2)$ generic sequence $\x_{\kappa}$ introduces a cofinal sequence of ordinals $c_\kappa$, which is Magidor generic for $\po^{\U}(\kappa,2)$ and has an infinite intersection with every set of $U_{\kappa,1}$. Therefore, $c_\kappa$
	contains a cofinal subsequence $c' = \la \nu_n \mid n < \omega\ra$ consisting  of $\CC[\bar{E_0}]$ potential generic points and is almost contained in every closed unbounded subset of $\kappa$ in $V[\G{\po}*\G{\Col}]$.
	Consequently, given a condition $d \in \CC[\bar{E_0}]$, we can form a $\CC[\bar{E}_0]$ generic club extending $d$,  by threading an $\omega$-sequence of locally generic conditions  $\bar{d}_{\nu_n} \subset \nu_n + 1$. It follows that $\po^{\U}(\kappa,2)$ projects to $\CC[\bar{E}_0]$ and so  $\po(\kappa,2)$ does too.  As before, we can reflect this assertion  on a $U_{\kappa,2}$ measure one set $E_0'(2) \subset E_0$, consisting of $\CC[\bar{E}_0]$ potentially generic ordinals. The same threading methods show that $\po^{\U}(\kappa,\tau)$ (and thus $\po(\kappa,\tau)$) projects to $\CC[\bar{E}_0]$, for all $1 < \tau \leq \lambda$, and implies that in $V$, there exists a subset $E_1' \in \bigcap_{1 \leq \tau < \lambda}U_{\kappa,\tau}$ which consists of $\CC[\bar{E_0}]$ potentially generic ordinals.  

	This concludes the absorption argument for $\CC_1 = \CC(0) =\CC[\bar{E_0}]$, which is  also the starting point for the argument for $\CC_2 = \CC(0)*\CC(1)$. We may assume that the set $E_1'$ of $\CC_1$-potentially generic ordinals is a subset of $E_1$.
	As mentioned above, the fact that $\bar{E_1}$ is fat stationary in $V[\G{\po}*\G{\col}]$ and every ordinal $\nu \in \bar{E_1} \setminus \cof(\lambda)$ is $\CC_1$ potentially generic implies 
	$\CC_2$ is $\kappa$-distributive.     
	We can therefore repeat the argument above, showing first that the supercompact Prikry forcing $\po(\kappa,1)$ subsumes $\CC_2$, and then, by in induction on $2 \leq \tau \leq \lambda$ that it is possible to thread $\CC_2\uhr \nu$-generics to produce a $\CC_2$ generic pair $\vec{C}\uhr 2 = \la C_0,C_1\ra$ in a  $\po^{\U}(\kappa,\tau)$ generic extension.
	This procedure can be repeated for every $\eta < \kappa^+$ using the same threading principle. 
	
	Let $\NG{\vec{C}}$ be a $\po(\kappa,\tau)$-name for a $\CC_{\eta}$ generic sequence, obtained by a threading process, as described above. Therefore, $\NG{\vec{C}}$ naturally assigns to every stem $t$ of a condition $p = \la t,T\ra$ a condition $\vec{d}_t \in \CC_{\eta}$ which is forced by $p$ to be an initial segment of $\NG{\vec{C}}$. We note that each $\po(\kappa,\tau)$ introduces many names of $\CC_{\eta}$ generic clubs, and thus, many different forcing projections to $\CC_{\eta}$. For example, we can decide to start threading the generic sequence $\vec{C}$ above a condition $\vec{d} \in \CC_{\eta}$, or change the ``local`` choices of $\CC_{\eta}\uhr \nu$ generic conditions $\vec{d}_\nu$, added to $\vec{d}_t$  when adding the ordinal $\nu$ to the stem $t$. 
	This flexible behavior of the threading procedure guarantees that the induced forcing projections from $\po(\kappa,\tau)$ to $\CC_\eta$ satisfy the following natural extension properties:
	
	\begin{enumerate}
		\item \textit{(Extendability of projections)}\\
		Let  $\eta_1< \eta_2$ and suppose $\NG{\vec{C}_{\eta_1}}$ is a $\po(\kappa,\tau)$ name of a $\CC_{\eta_1}$ generic sequence for some $1 \leq \tau \leq \lambda$, and let $\pi_{\eta_1} = \pi_{\NG{\vec{C}_{\eta_1}}}$ be its induced forcing projection 
		from $\po(\kappa,\tau)$ to the Boolean completion of $\CC_{\eta}$.
		For every $p \in \po(\kappa,\tau)$ and a condition $\vec{d}_2 \in \CC_{\eta_2}$ such that $\vec{d}\uhr \eta_1 \in \CC_{\eta_1}$ is compatible with $\pi_{\eta_1}(p)$, there exists an extension of $\pi_{\eta_1}$  to a projection 
		$\pi_{\eta_2} = \pi_{\NG{\vec{C}_{\eta_2}}}$ to $\CC_{\eta_2}$ such that 
		$\pi_{\eta_2}(p)$ and $\vec{d}$ are compatible.
		
		\item \textit{(Compatibility of direct extensions and quotients)}\\
		Let $(q,\name{g})$ be a condition of the preparation forcing $\po*\Col$. Suppose that
		$\name{p} = \la \can{t},\name{T}\ra$, $\name{\vec{d}}$, and  $\NG{\vec{C}}$ are $\po*\Col$ names
		of conditions (in $\po(\kappa,\tau)$ and $\CC_{\eta}$ respectively) and a $\po(\kappa,\tau)$-name of a $\CC_{\eta}$ generic sequence.
		If $(q,\name{g})$ forces  
		$\vec{d}$ is compatible with the condition $\vec{d}_t \in \CC_{\eta}$, determined by the stem $t$ of $p$ and $\NG{\vec{C}}$, then there is a condition 
		$(q',\name{g}')$ extending $(q,\name{g})$,  and a point $x$ such that $(q',\name{g'})$ forces $x \in \suc_{\name{T}}(t)$ and $\vec{d}_{t \fr \la x\ra}$ extends $\vec{d}$.
		By a standard genericity argument, it follows that
		the direct extension order of $\po(\kappa,\tau)$ is compatible with the projections induced by $\NG{\vec{C}}$. Namely, 
		if $p' = \la t,T'\ra$ belongs to the quotient $\po(\kappa,\tau)/\la \NG{\vec{C}}, \vec{C}_\eta\ra$ and $p = \la t,T\ra$ is a direct extension of $p'$, then $p$ is also a member of the quotient. 
	\end{enumerate}

	\noindent\textbf{The Prikry posets $\qo^*(\tau)$ - } 
	Our final forcing is $\qo = (\po * \col) * \CC$ where $\CC = \CC_{\kappa^+}$ is the iteration limit of $\CC_{\eta}$, $\eta < \kappa^+$. 
	The restriction on the support to be of size $<\kappa$ guarantees $\CC$ satisfies $\kappa^+$.c.c, and therefore a sequence of clubs $\vec{C} = \la C_i \mid i < \kappa^+\ra$ is generic if and only if its initial segments $\vec{C}\uhr \eta$, $\eta < \kappa^+$ are generic for $\CC_{\eta}$.
	For every $\tau$, $1 \leq \tau \leq \lambda$, let $\qo^*(\tau) = (\po*\col)* \po(\kappa,\tau) = (\po*\po(\kappa,\tau))*\col$.
	The forcing $\po*\po(\kappa,\tau)$ is clearly of Prikry-type with $\lambda$-closed direct extension order, and by trivially extending the direct extension order to include the order of the collapse part $\col$, we obtain a desirable Prikry-type forcing structure on $\qo^*(\tau)$. 
	
	We claim $\qo^*(\tau)$ absorbs $\qo$. It is clearly sufficient to show $\po(\kappa,\tau)$ absorbs $\CC$, which we prove by an induction on $\tau$. For $\tau = 1$, recall $\po(\kappa,\tau)$ adds an $\omega$ sequence $\la \eta_n \mid n < \omega\ra$ which is cofinal in $\kappa^+$.  Then, using the fact $\po(\kappa,\tau)$ projects to $\CC_{\eta_n}$ for each $n < \omega$ and the extendability of projections, it is routine to form a sequence $\la \vec{d}_m \mid m < \omega\ra$ whose restriction to $\CC_{\eta_n}$ for each $n < \omega$ is generic. Hence, the sequence form a $\CC$ generic filter. 
	For $\tau > 1$, the forcing $\po(\kappa,\tau)$ adds a supercompact Magidor sequence $\x_{\kappa} = \la x_i \mid i < \tau'\ra$ (where $\tau'$ is the ordinal exponent $\omega^\tau$) which covers $\kappa^+$ and therefore induces a cofinal sequence $\la \eta_i \mid i < \tau'\ra$. By the induction hypothesis, we can form an increasing sequence $\la \bar{d_i} \mid i < \tau'\ra$ so that $\bar{d_i} \in \CC_{\eta_i}$ is generic for $\CC_{\eta_i}\uhr \kappa_{x_i}$. The local genericity of the $\bar{d_i}$ and the fact that $\eta_i$ is cofinal in $\kappa^+$ imply the sequence generates a $\CC$ generic filter. \\

	\noindent\textbf{The precipitous of the nonstationary ideal in $V[\G{\qo}]$ -}
	For each $\eta < \kappa^+$, let $\qo\uhr \eta$ denote the sub-forcing $(\po*\col)*\CC_{\eta}$ of $\qo = (\po*\col)*\CC$. 
	We denote the restricted nonstationary ideal on $\kappa$ to the family of stationary sets $S \subset \kappa \cap \cof(<\lambda)$ by 
	$\NS_{\kappa}\uhr \cof(<\lambda)$. It is shown in 
	\cite{Gitik-iteration} that for every $\eta < \kappa^+$, the restriction of $\NS_{\kappa}\uhr \cof(<\lambda)$ to sets
	in $V[\G{\qo_\eta}]$ is characterized by the following. For every subset of $\kappa \cap \cof(<\lambda)$, $Z = \name{Z}_{\G{\qo_\eta}}$ in $V[\G{\qo_\eta}]$, $Z$ is nonstationary if and only if for every ordinal $\tau$, $1 \leq \tau < \lambda$, there are conditions $(p,g) \in \G{\po}*\G{\col}$ and $\vec{d} \in \G{\CC_{\eta}}$ such that as conditions in $j_{\tau}(\po*\col)$ and $j_\tau(\CC_{\eta})$ respectively, 
	\[(p,g) \force \forall \vec{C}_\eta \geq \vec{d} \text{ generic over } M_\tau[\NG{\po}*\NG{\col}], \quad \vec{d}_{\vec{C_\eta}}^\tau \force_{j_\tau(\CC_{\eta})} \can{\kappa} \not\in j_\tau(\name{Z}),\]
	where $\vec{d}_{\vec{C_\eta}}^\tau$ is the $j_\tau(\CC_{\eta})$ condition, obtained from $\vec{C}_\eta = \la C_i \mid i < \eta\ra$ by replacing each $C_i$ with $\bar{C_i} = C_i \cup \{\kappa\}$, and changing the indexing of $\vec{C}_\eta$ from $i \in \eta$ to $i \in j_\tau``\eta$. Moreover, since in our arguments above, every $\CC_{\eta}$-generic sequence was constructed from the $j_\tau(\po*\col)$ sub-forcing $j_\tau(\po)_{\kappa+1}*j_\tau(\col)_\kappa = \po*\po(\kappa,\tau)*\col$, we may assume that 
	$(p,g)$ above forces $\vec{C}_\eta$ is a $j_\tau(\po)\uhr\kappa+1 * j_\tau(\col)_\kappa$ name of a $\CC_\eta$ generic sequence.
	
	Now, by a standard genericity argument, for every stationary set $S \subset \kappa \cap \cof(<\lambda)$ in $V[\G{\qo_\eta}] = V[\G{\po}*\G{\col}*\G{\CC_{\eta}}]$, there are $(p',g') \in j_\tau(\po*\col)$ and a name $\NG{\vec{C}_\eta}$, which satisfy the following conditions:
	
	\begin{itemize}
		\item $\NG{\vec{C}_\eta}$ is a $\po*\po(\kappa,\tau)*\col$ name of a $\CC_{\eta}$ generic sequence.
		\item $(p',g') \force_{j_\tau(\po*\col)} \vec{d}^\tau_{\NG{\vec{C}_\eta}} \force \can{\kappa} \in j_\tau(\name{S})$    
		\item 
		$(p',g')$ and $\NG{\vec{C}_\eta}$ are compatible with the generic information of $\G{\po}$, $\G{\col}$ and $\G{\CC_{\eta}}$. Namely, $(p',g')$ is a member of the quotient forcing $j_\tau(\po*\col)/ (\G{\po*\col}*\la \G{\CC_{\eta}} \NG{\vec{C}_\eta}\ra)$\footnote{see Section \ref{section-0} for explanation of the notation.}.
	\end{itemize}
	
	It clearly follows that by forcing with the quotient given above, we can generically extend the embedding $j_\tau : V \to M_\tau$ to have domain $V[\G{\qo_\eta}]$. Therefore, $V[\G{\qo_\eta}]$ cannot contain a counter example for the precipitousness of $\NS_{\kappa}\uhr \cof(<\lambda)$. But a counter example for precipitousness is just an $\omega$ sequence of functions in $\kappa^\kappa$ and $\CC$ satisfies $\kappa^+$.c.c, so $\NS_{\kappa}\uhr \cof(<\lambda)$ must be precipitous in $V[\qo]$. 
	
	This concludes the description of the poset $\qo$ and the key arguments in \cite{Gitik-iteration}. We turn to the proof of Proposition \ref{prop4.1}. The idea is to show that quotient posets $j_\tau(\qo)/\qo$, $\tau < \lambda$ are subsumed by Prikry-type forcings, which allow us to generically extend the embedding $j_\tau$ to domain $V[\G{\qo}]$. \\
	
	\emph{proof.} \textbf{(Proposition \label{prop 4.1}).}
	We already proved that the poset $\qo = (\po*\col)*\CC$ satisfies the first three conditions in the statement of Proposition \ref{prop 4.1}.
	It remains to prove that in $V[\G{\qo}]$, every stationary set $S \subset \kappa \cap \cof(<\lambda)$ is positive with respect to a $\lambda$-Prikry-closed nonstationary ideal on $\kappa$.
	
	Let $S \subset \kappa \cap \cof(<\lambda)$ be stationary
	in $V[\G{\qo}] = V[\G{\po}*\G{\col}*\G{\CC}]$, and fix $\eta < \kappa^+$ for which there exists a $\CC_{\eta}$ name $\name{S}$ for $S$.
	For ease of notation, we will not distinguish between the $\CC$ generic filter $\G{\CC}$ and 
	its induced generic club sequence $\vec{C} = \la C_i \mid i < \kappa^+\ra$.
	By the arguments described above, concerning the precipitousness of the nonstationary ideal in $V[\G{\qo}]$, there is an ordinal $\tau$, $1 \leq \tau < \lambda$, a $\po*\po(\kappa,\tau)*\col$-name $\NG{\vec{C}_\eta}$ of a $\CC_{\eta}$-generic sequence, and a condition $(p',g')$ in the quotient forcing $j_\tau(\po*\col)/ (\G{\po}*\G{(\col)}* \la \vec{C}\uhr\eta,\NG{\vec{C}_\uhr\eta}\ra)$, so that
	\[ p'*g'*\vec{d}^\tau_{\vec{C_\eta}} \force_{j_\tau(\qo_\eta)} \can{\kappa} \in j_\tau(\name{S}).\]
	
	By naturally identifying $\qo_\eta$ as a sub-forcing of $\qo$, we may consider     $p'*g'*\vec{d}^\tau_{\vec{C_\eta}}$ as a condition of $j_\tau(\qo)$. Now, $\NG{\vec{C_\eta}}$ is a $\po*\col*\po(\kappa,\tau)$-name. Recall $\po*\col*\po(\kappa,\tau)$ projects to $\CC$,  and moreover, by the extendability properties of the projections of $\po(\kappa,\tau)$ to $\CC_{\eta}$, there exists a 
	$\po*\col*\po(\kappa,\tau)$-name $\NG{\vec{C}}$ for a $\CC$-generic sequence which extends $\NG{\vec{C_\eta}}$, so that
	$(p',g')$ belongs to the extended quotient 
	$j_\tau(\po*\col)/ (\G{\po}*\G{(\col)}* \la \vec{C},\NG{\vec{C}}\ra)$. 
	
	Let $\la C_i \mid i < \kappa^+\ra$ be an enumeration  of $\vec{C}$. Note that function $\vec{d}_{\vec{C}}^\tau = \{ \la j_\tau(i), C_i \cup \{{\kappa}\} \ra \mid i < \kappa^+\ra$ belongs to $M_\tau[\G{\qo}]$ since $M_\tau^{\kappa^+} \subset M_\tau$.
	
	Working in $M_\tau[\G{\qo}]$, consider the poset 
	\[\RR = j_\tau(\po*\col)/ (\G{\po}*\G{(\col)}* \la \vec{C},\NG{\vec{C}}\ra) \thinspace\thinspace * \thinspace\thinspace  j_\tau(\CC) \] 
	and its condition 
	\[r = p'*g' \thinspace\thinspace* \thinspace\thinspace\vec{d}_{{\vec{C}}}^\tau. \] 
	
	Define, in $V[\G{\qo}]$, the ideal $I_{\ro,r}$
	to be the family of all subsets $Z = \name{Z}_{\G{\qo}}$ of $\kappa$  such that 
	$r \force_\RR \can{\kappa} \not\in j_\tau(\name{Z})$.
	It is easy to see that $I$ is a non-principal $\kappa$-complete ideal on $\kappa$ and that $S$ belongs to its dual filter. 
	Moreover, it is easy to see that $r$ is a master condition with respect to $\G{\qo}$\footnote{Namely, $r$ extends $j_\tau(q)$ for every  $q \in \G{\qo}$.}, and consequently, if $C = \name{C}_[\G{\qo}]$ is a closed unbounded subset of $\kappa$ then $r$ forces $j_\tau(\name{C})$ is closed unbounded, and that $j_\tau(\name{C}) \cap \can{\kappa} = \name{C}$. In particular $r$ forces $\kappa$ is a limit point of $j_\tau(\name{C})$ and thus that it belongs to $j_\tau(\name{C})$.
	
	It remains to show $I_{\ro,r}$ is $\lambda$-Prikry-closed.
	Recall $\qo$ is subsumed by the Prikry-type forcing $\qo^* = \po*\col*\po(\kappa,\lambda)$.Therefore, $\RR$ is subsumed by the poset $\ro^* = j_\tau(\po*\col)/ (\G{\po}*\G{(\col)}* \la \vec{C},\NG{\vec{C}}\ra) \thinspace\thinspace * \thinspace\thinspace  j_\tau(\po(\kappa,\lambda))$.
	We may also assume there are conditions $r^* \in \ro^*$ and a forcing projection $\pi :\ro^* \to \ro$ such that $\pi(r^*) = r$. 
	
	Since $\po$ and $\col$ are identified here as sub-forcings of $j_\tau(\po)$ and $j_\tau(\col)$, respectively, $\ro^*$ we can identify $\ro^*$ with the quotient $\ro'/\la \vec{C},\NG{\vec{C}}\ra$, where \[ \ro' = j_\tau(\po)/\G{\po} * \left(j_\tau(\po(\kappa,\lambda)) \times j_\tau(\col)/\G{\col}\right).\] 
	
	We endow $\ro^*$ with the direct extension order $\leq^*$ of $\ro'$, which is the one inherited from the natural direct extension order of $j_\tau(\po)/\G{\po} * j_\tau(\po(\kappa,\lambda))$ and the usual order relation of 
	$j_\tau(\col)/\G{\col}$. It is easy to see that $\ro'$ with this direct order relation is $\lambda$-closed and of Prikry-type. 
	Furthermore, by the compatibility of direct extensions with $\NG{\vec{C}}$ induced quotients, the $\ro'$ quotient $\ro^*$ is compatible with $\leq^*$ in the sense that for every $t \in \ro^*$ and $t' \in \ro'$, if $t' \geq^* t$ then $t' \in \ro^*$.
	It follows that $\ro^*$ with $\leq^*$ forms a $\lambda$-closed-Prikry-type forcing. 
	
	Next, by a standard argument about ideals derived from elementary emebeddings, the forcing $\ro$ adds a generic filter for forcing with the positive sets (see \cite{For-HB}). Hence, $\ro^*$ projects to a dense subset of $I_{\ro,r}^+$, and we can translate the direct extension sub-order of $\ro^*$ to a sub-order $\leq^*_I$ of the ideal relation $\leq_{I_{\ro,r}} \uhr D$. It follows at once that $\leq^*_I$ is $\lambda$-closed. Suppose that $B \in D$ corresponds to some $b \in \ro^*$, and $\name{f}$ is a $\qo$-name for a function $f : B \to 2$. Then $b \force_{\ro^*} j_\tau(\name{f}) : j_\tau(\kappa) \to 2$ has a direct extension $b^* \geq^* b$ which forces $j_\tau(\name{f})(\can{\kappa}) = i^*$ for some $i^* \in 2$. It follows there exists $B^* \leq^* B$ which corresponds to $b^*$ such that $f``B^* = \{i^*\}$. \qed{Proposition \ref{prop4.1}}\\
	\qed{Theorem \ref{thm1.3}}
	
\providecommand{\bysame}{\leavevmode\hbox to3em{\hrulefill}\thinspace}
\providecommand{\MR}{\relax\ifhmode\unskip\space\fi MR }
\providecommand{\MRhref}[2]{%
	\href{http://www.ams.org/mathscinet-getitem?mr=#1}{#2}
}
\providecommand{\href}[2]{#2}

\end{document}